\documentclass[10pt]{amsart}

\usepackage{txfonts}
\usepackage[T1]{fontenc}
\usepackage{fancyhdr}

\usepackage[utf8]{inputenc} 
\usepackage{geometry} 
\usepackage{booktabs} 
\usepackage{array} 
\usepackage{paralist} 
\usepackage{verbatim} 
\usepackage{tikz}
\usetikzlibrary{arrows}
\usetikzlibrary{decorations.markings}
\usepackage{subfig}
\usepackage{tabularx}
\usepackage{enumitem}
\usepackage{amsmath,amsthm,amscd}
\usepackage{amssymb,amsfonts}
\usepackage{fancyhdr} 
\usepackage{pgf,tikz}
\usepackage{caption}
\usepackage{tikz-cd}
\usepackage{tikzscale}
\usetikzlibrary{patterns}
\usepackage{rotating}
\usepackage{xcolor}
\usepackage{lscape}
\usepackage{xspace}
\usepackage{xfrac}
\usepackage{blindtext}
\usepackage{bm}
\usepackage[ruled,vlined]{algorithm2e}

\usetikzlibrary{decorations.markings}
\usetikzlibrary{patterns}
\usetikzlibrary{calc}
\usepackage{soul}

\usepackage{tikz}
        \usetikzlibrary{patterns,hobby}
        \usepackage{pgfplots}
        \pgfplotsset{compat=1.6}

\usepackage{faktor} 
\usepackage{pgf,tikz}
\usetikzlibrary{arrows.meta}
\usetikzlibrary{decorations.markings}
\usetikzlibrary{patterns}

\usepackage{color}   
\usepackage{hyperref}
\hypersetup{
    colorlinks=true, 
    linktoc=all,     
    linkcolor=blue,  
    citecolor=blue,
    filecolor=blue,
    urlcolor=blue
}

\usepackage[T1]{fontenc}
\usepackage{hyperref}
\usepackage{nameref}

\newtheorem{thm}{Theorem}[section]
\newtheorem{cor}[thm]{Corollary}
\newtheorem{prop}[thm]{Proposition}
\newtheorem{lem}[thm]{Lemma}
\theoremstyle{definition}
\newtheorem{defn}[thm]{Definition}
\theoremstyle{remark}
\newtheorem{rmk}[thm]{Remark}
\theoremstyle{definition}

\theoremstyle{definition}
\newtheorem{ex}[thm]{Example}
\theoremstyle{definition}

\numberwithin{equation}{section}
\title[Bounds on saddle connections for flat spheres]{Bounds on saddle connections for flat spheres}

\author{Kai Fu}
\address[Kai Fu]{Max Planck Institute for Mathematics in the Sciences, Leipzig, Germany}
\email{kai.fu@mis.mpg.de}

\author{Guillaume Tahar}
\address[Guillaume Tahar]{Beijing Institute of Mathematical Sciences and Applications, Huairou District, Beijing, China}
\email{guillaume.tahar@bimsa.cn}

\date{\today}
\keywords{Flat metric, Conical singularities, Saddle connection, Delaunay triangulation, Curvature gap, Flat annulus}

\begin{document}
\begin{abstract}
We consider a flat metric with conical singularities on the sphere. Under the assumption that no partial sum of angle defects is equal to $2\pi$, we draw on the geometry of immersed disks to obtain an explicit upper bound on the number of saddle connections with at most $k$ self-intersections. Additionally, we establish an upper bound on their lengths for a surface with a normalized area. Finally, we apply these bounds to the counting of singular trajectories in irrational polygonal billiards.
\end{abstract}
\maketitle
\setcounter{tocdepth}{1}
\tableofcontents

\section{Introduction}

A \textit{flat sphere} is a topological sphere endowed with flat metric in the complement of finitely many conical singularities. Equivalently, it can be described by a collection of Euclidean triangles glued isometrically along a pairing of edges (see \cite{Ta1} or \cite{Tr} for details). When conical angles are rational multiples of $\pi$, flat metrics on the sphere are induced by $k$-differentials (see \cite{BCGGM} for an extensive study of these differentials and their moduli spaces).

\begin{defn}
A \textit{trajectory} in a flat sphere is a geodesic without singularities in the interior. It is \textit{simple} if there is no self-intersections in the interior. A \textit{saddle connection} is a finite trajectory whose endpoints are conical singularities.
Self-intersections are counted with their multiplicities (see Definition~\ref{defn:transverse}).
\end{defn}

In this paper, we are interested in the counting of saddle connections on flat spheres. The key ingredients of our bounds are:
\begin{enumerate}
    \item the notion of the \textit{curvature gap}, measuring the obstruction to realize a partition of the set of conical singularities in two sets of equal total curvatures, see Section~\ref{sec:curvature};
    \item the geometric properties of \textit{flat annuli}, analogs of flat cylinders of translation surfaces, see Section~\ref{sec:annuli};
    \item locally isometric immersions of disks and subsequent \textit{Delaunay triangulations}, see Section~\ref{sec:Delaunay}.
\end{enumerate}

We focus on the sphere case only in this paper because there is no reasonable generalization for our notion of the curvature gap in higher genus cases. We further restrict to the case of a positive curvature gap. This assumption can be regarded as generic among flat spheres: indeed, the vanishing of the curvature gap imposes additional linear constraints on the cone angles; see Section~\ref{sec:curvature}.

\begin{thm}\label{thm:MAIN}
In a flat sphere with $n$ conical singularities and a curvature gap of $\delta>0$, the number of saddle connections with at most $k$ self-intersections is at most $(3n-6)2^{s}$ where 
$s=\frac{20n(n-1)\sqrt{k} + 20n}{\delta}$.
\par
In the case of simple saddle connections, the upper bound can be improved to $\frac{1}{(3n-7)!}(\frac{5n}{\delta}+3n-7)^{3n-6} + 3n-6$.
\end{thm}

\begin{rmk}
Similar estimates also hold for regular closed geodesics on flat spheres; see Corollary~\ref{cor:countclosedgeodesics}.
\end{rmk}

Theorem~\ref{thm:MAIN} is proved in Section~\ref{sub:Self}. It follows essentially from an estimate on combinatorial lengths of simple trajectories in a flat sphere (see Proposition~\ref{prop:combinatorial}) and some topological considerations about self-intersecting arcs.\newline

According to the counting of simple saddle connections in flat annuli (see Section~\ref{sub:example}), the curvature gap $\delta$ cannot be eliminated from the upper bound in Theorem~\ref{thm:MAIN}.\newline

We emphasize that the bounds in Theorem~\ref{thm:MAIN} are uniform in the following sense: as functions of $k$, their constants depend only on the number of singularities and on the curvature gap. This property will play an important role in subsequent work~\cite{Fu2024,Fu2025}, which we briefly discuss at the end of the introduction.\newline

On flat surfaces, one usually studies counting problems for saddle connections with bounded metric length; in general this is a difficult problem. Partial progress has been made in special settings; for instance, sub-exponential upper bounds are known in certain cases of flat spheres; see~\cite{Kat,Sch}. When all cone angles are rational multiples of~$\pi$, the situation is much better understood: in particular, Masur proved quadratic upper and lower bounds for the number of periodic geodesics of bounded length~\cite{Ma1,Ma2}. We refer to~\cite{AthreyaMasur2024,Filip2024,Wright2015,Zor} for further background and developments in this setting.

Theorem~\ref{thm:MAIN} provides a way to approach counting saddle connections with bounded length. This motivates the study of the relation between metric length and self-intersection number.

\begin{thm}\label{thm:MAIN2}
In a flat sphere of unit area with $n$ conical singularities and a curvature gap $\delta>0$, the metric length of any trajectory with at most $k$ self-intersections is at most 
$
\frac{40n(n-1)\sqrt{k}+40n}{\delta\sqrt{\pi}}
+
\frac{20n(n-1)\sqrt{k}+20n}{\delta^{3/2}\sqrt{2\pi}}
$.
\par
The bound in the case of a simple trajectory can be slightly improved to $\frac{10n}{\delta\sqrt{\pi}}+\frac{5n}{\delta^{3/2}\sqrt{2\pi}}$.
\end{thm}

The proof of Theorem~\ref{thm:MAIN2} is given in Section~\ref{sub:length}. The bound relies on an estimate for the lengths of edges in Delaunay triangulations of flat spheres (see Lemma~\ref{lem:Delaunay}). In particular, it does not depend on the systole\footnote{The systole is the length of the shortest simple saddle connection.} of the surface. Delaunay triangulations have previously been used to study the distribution of saddle connections on flat surfaces; see~\cite{ACM} for example.

Lower bounds on the metric length in terms of the self-intersection number on flat spheres will be studied further in~\cite{Fu2024}.\newline

For comparison, on hyperbolic surfaces one has lower bounds on the metric length in terms of the self-intersection number; see~\cite{Ba,BaPV}. In contrast, Theorem~\ref{thm:MAIN2} provides upper bounds in the setting of flat spheres. Such upper bounds do not hold in general for hyperbolic surfaces. Indeed, when $3g-3+n>0$, there exist simple closed geodesics of arbitrarily large length on hyperbolic surfaces.\newline

From a broader perspective, our work fits into a comparison between spherical, flat, and hyperbolic geometries. On spherical surfaces, geodesics necessarily self-intersect, since any two geodesics intersect on $\mathbb{S}^2$. On hyperbolic surfaces, self-intersections arise from the action of the fundamental group on $\mathbb{H}^2$. Flat surfaces without singularities admit no self-intersecting geodesics. In contrast, in the flat setting studied here, conical singularities together with curvature constraints force geodesics to self-intersect. Given the extensive study of flat surfaces, it is natural to study self-intersections of geodesics.\newline

As a further consequence of the above length estimate on Delaunay edges, one also obtains an essentially sharp upper bound on the diameter of flat spheres.

\begin{cor}\label{cor:diameter}
In a flat sphere $X$ of unit area with $n$ conical singularities and a curvature gap $\delta>0$, the diameter of $X$ as a metric space is at most $(n+1)(\frac{2}{\sqrt{\pi}}+\frac{1}{\sqrt{2\pi\delta}})$.
\end{cor}

\bigskip

We do not address the sharpness of the upper bounds in Theorem~\ref{thm:MAIN} and Theorem~\ref{thm:MAIN2}, but briefly comment on two aspects of these bounds. In Theorem~\ref{thm:MAIN}, the upper bounds grow exponentially in $\sqrt{k}$, where $k$ denotes the self-intersection number. In view of the sub-exponential upper bounds known for counting saddle connections with bounded metric length, one may expect that this growth in $k$ is not optimal. For the length bound in Theorem~\ref{thm:MAIN2}, Corollary~\ref{cor:diameter} suggests that the dependence on the curvature gap $\delta$ could potentially be of order $\delta^{-1/2}$. Clarifying the sharpness of these bounds remains an interesting question.\newline

As an application, these estimates apply in particular to polygonal billiards. A \textit{billiard path} in a polygonal billiard is a continuous path consisting of straight segments inside the polygon and reflections on the boundary, where the angles of incidence and reflection are equal. We assume that a billiard path stops when it hits a vertex. A billiard path is called a \textit{generalized diagonal} if it joins two vertices of the polygon. It is called \textit{periodic} if it returns to its starting point with the same direction. We also define the notion of a \textit{curvature gap} for polygonal billiards; see Section~\ref{sec:billiard}.

A polygonal billiard is called \textit{rational} if all angles of the polygon are rational multiples of~$\pi$; otherwise, it is called \textit{irrational}. Counting generalized diagonals and periodic billiard paths is a classical problem. The case of rational polygons is well studied; for an overview, see the surveys and books~\cite{AthreyaMasur2024,Filip2024, MasurTabachnikov2002,McMullen2023,Wright2015,Zor}.

Very little is known about irrational billiards. Partial results show, for instance, that every triangle with angles at most $112.3^\circ$ admits a periodic billiard path; see~\cite{RSch06, RSch08, tokarsky2018point}. The existence of a regular periodic billiard path for general polygons remains an open problem.

For generalized diagonals, Katok proved in~\cite{Kat} that the number of generalized diagonals grows sub-exponentially. Scheglov later provided in~\cite{Sch} an explicit sub-exponential upper bound for almost every triangle, but no general explicit bound is known for arbitrary polygons. For the lower bound, Hooper showed in~\cite{HooperLowerBounds} that for any $k \in \mathbb{N}$, there exists an irrational polygon for which the number of periodic billiard paths of length at most $R$ grows at least like $R \log^k R$.

From Theorem~\ref{thm:MAIN}, we deduce an explicit upper bound for a closely related counting problem involving generalized diagonals and periodic billiard paths; see Section~\ref{sec:billiard}. Care is needed when defining the self-intersection number for billiard paths. A precise definition is provided in Section~\ref{sec:intersectionnumberbilliardpaths}.

\begin{thm}\label{thm:MAIN3}
In a (possibly irrational) polygonal billiard with $n$ vertices and a curvature gap $\delta > 0$, the number of generalized diagonals with at most $k$ self-intersections is bounded from above by $(3n-6)2^{s}$ where $s=\frac{20n(n-1)\sqrt{k} + 20n}{\delta}$. The number of maximal families of parallel periodic billiard paths with at most $k$ self-intersections satisfies the same bound.
\par
Moreover, if the polygon is of area $1$, the Euclidean length of any generalized diagonal with at most $k$ self-intersections is at most $
\frac{40n(n-1)\sqrt{2k}+40\sqrt{2}n}{\delta\sqrt{\pi}}
+
\frac{20n(n-1)\sqrt{k}+20n}{\delta^{3/2}\sqrt{\pi}}
$.
\end{thm}

Beyond applications to polygonal billiards, the results of this article fit into a broader program; see~\cite{Fu2024,Fu2025}. In~\cite{Fu2024}, the relation between self-intersection numbers and metric lengths of saddle connections on flat spheres is studied further. Combined with Theorem~\ref{thm:MAIN2}, this yields uniform estimates on the metric length of saddle connections. These, in turn, allow Theorem~\ref{thm:MAIN} to provide uniform bounds for counting saddle connections with bounded metric length on flat spheres.

In~\cite{Fu2025}, averaging formulas for counting functions over the moduli space of flat spheres are studied. The uniform bounds provided by Theorem~\ref{thm:MAIN} are a key input for establishing integrability properties of these counting functions. 
More generally, the results of the present article are used in the derivation of averaging formulas.
\newline

\paragraph{\bf Acknowledgements.} The first author would like to thank his advisors, Vincent Delecroix and Elise Goujard, for introducing him to the flat geometry and many valuable discussions. He wants to acknowledge Yitwah Cheung for his invitation and the opportunity to conduct this research at Tsinghua University. The research of the first author is supported by ANR MoDiff (ANR-19-CE40-0003). The research by the second author is supported by the Beijing Natural Science Foundation (Grant IS23005) and the French National Research Agency under the project TIGerS (ANR-24-CE40-3604). The authors would also like to acknowledge Dmitri Panov and the anonymous referee for valuable remarks.

\section{Discrete curvature}\label{sec:curvature}

For a conical singularity $M$ of an angle $\theta_{M}$ in a flat sphere $X$, we define its \textit{discrete curvature} $k_{M}=\frac{2\pi-\theta_{M}}{2\pi}$. In particular we have $k_{M} \in ]-\infty,1[$. The classical Gauss-Bonnet formula writes $\sum\limits_{M \in X} k_{M} =2$.

\subsection{Curvature gap}

The obstruction to realize a partition of the set of conical singularities in two sets of equal total curvature is measured by the \textit{curvature gap}.

\begin{defn}\label{defn:gap}
For a flat sphere $X$ with a finite set $\Lambda$ of conical singularities, we define the \textit{curvature gap} of $X$ as 
$\inf\limits_{I \subset \Lambda} |1-\sum\limits_{i \in I} k_{i}|$.
\end{defn}

A combinatorial analog of the curvature gap has been introduced in Section~6 of \cite{Ta} to characterize strata of quadratic differentials where the corresponding half-translation surfaces can have infinitely many saddle connections.
\par
Indeed, any simple closed geodesic cuts out the sphere into two connected components, each containing a total discrete curvature of $1$ by Gauss-Bonnet formula. It follows that a flat sphere satisfying $\delta>0$ does not contain any simple closed geodesic (and therefore no flat cylinder).

Note that the vanishing of the curvature gap imposes affine conditions of the form
\[\sum_{i \in I} k_i = 1
\]
for some $I \subset \Lambda$, inside the space of curvatures $\{ (k_i)_{i\in\Lambda}\mid \, \sum_{i\in\Lambda} k_i = 2\}.$
In particular, the locus of zero curvature gap has Lebesgue measure zero in the space of curvatures.

\subsection{Upper bound on the curvature gap}

In this Section, we prove a sharp bound on the curvature gap and characterize completely the equality case.

\begin{lem}\label{lem:curvaturegap}
The curvature gap $\delta$ of any flat sphere satisfies $\delta \leq \frac{1}{3}$. The equality holds if and only if all of conical angles belong to $2\pi+\frac{4\pi}{3}\mathbb{Z}$.
\par
Equivalently, the equality case for the curvature gap corresponds to flat metrics of finite area induced by meromorphic cubic differentials with singularities of even order.
\end{lem}

\begin{proof}
In what follows we consider flat spheres with $n$ conical singularities of discrete curvatures $k_{1},\dots,k_{n} \in ]-\infty,1[$. Without loss of generality we assume that $k_{1},\dots,k_{s} >0$ for some rank $s$ while $k_{s+1},\dots,k_{n} \leq 0$. Since $\sum\limits_{i=1}^{n} k_{i} =2$, at least three conical singularities have positive curvature. In other words, we have $s \geq 3$.
\par
We first prove that for any such flat sphere satisfying $\delta \geq \frac{1}{3}$, all the positive discrete curvatures are equal to $\frac{2}{3}$ and the curvature gap is equal to $\frac{1}{3}$.
\par
By hypothesis, the positive discrete curvatures $k_{1},\dots,k_{s}$ are contained in the interval $]0,\frac{2}{3}]$. Moreover, we have  $\sum\limits_{i=1}^{s} k_{i} \geq 2$. Since no partial sum $\sum\limits_{i=1}^{l} k_{i}$ belongs to $]\frac{2}{3},\frac{4}{3}[$, there is a rank $l \leq s$ such that $\sum\limits_{i=1}^{l-1} k_{i} \leq \frac{2}{3}$ and $\sum\limits_{i=1}^{l} k_{i} \geq \frac{4}{3}$. We deduce that $k_{l} \geq \frac{2}{3}$ and therefore $k_{l}=\frac{2}{3}$. By definition, $\delta\leq|1-k_l| = \frac{1}{3}$, so $\delta = \frac{1}{3}$. In particular, there is no flat sphere satisfying $\delta > \frac{1}{3}$.
\par
Now, for any positive discrete curvature $k_{i}$ with $i \neq l$, we have $\frac{2}{3} < k_{l}+k_{i} \leq \frac{4}{3}$. Since $\delta = \frac{1}{3}$, we have $k_{l}+k_{i}=\frac{4}{3}$ and $k_{i}=\frac{2}{3}$. Every positive number among $k_{1},\dots,k_{n}$ is equal to $\frac{2}{3}$. 
\par
It remains to characterize completely the equality case. We already know that conical singularities of positive curvature have a conical angle equal to $\frac{2\pi}{3}$. We will prove that conical singularities of negative curvature have a discrete curvature that is an integer multiple of $\frac{2}{3}$ and therefore a conical angle that belong to $2\pi+\frac{4\pi}{3}\mathbb{N}$.
\par
We have $k_{1}=\dots=k_{s}=\frac{2}{3}$ and $k_{t} \leq 0$ for any $t>s$. For any such $k_{t}\leq 0$, consider the partial sums of the form $k_t + \sum\limits_{0\leq i \leq l} k_{i}=k_{t}+\frac{2l}{3}$ for $l$ satisfying $0 \leq l \leq s$. Since the sequence of these numbers increases by $\frac{2}{3}$ at each step, there exists $l$ such that
$$
k_t + \frac{2l}{3} \in  \left[\frac{2}{3},\frac{4}{3}\right].
$$
However, if $k_t + \frac{2l}{3}$ lies in the open interval $]\frac{2}{3},\frac{4}{3}[$, this would contradict $\delta=\frac{1}{3}$. Therefore, $k_t + \frac{2l}{3}$ must be $\frac{2}{3}$ or $\frac{4}{3}$. It follows that $k_{t}$ is an integer multiple of $\frac{2}{3}$.
\par
Then the corresponding flat metrics have conical angles in $2\pi+\frac{4\pi}{3}\mathbb{Z}$ at all singularities. Thus, the holonomy of the metric for a simple loop around any conical singularity is either trivial or a rotation of order three. In genus zero, the holonomy of the metric is generated by the holonomy of the loops around the conical singularities. Then, for any developing map $f$ of the metric, the cubic differential $(df)^{\otimes 3}$ is globally defined on the flat sphere. At a conical singularity of angle $2\pi(1-k_{i})$, $(df)^{\otimes 3}$ is a singularity of order $-3k_{i}$ (see \cite{BCGGM} for the background on differentials of higher orders). Since every discrete curvature $k_{i}$ is an integer multiple of $\frac{2}{3}$, every singularity of $(df)^{\otimes 3}$ is of even order and strictly larger than $-3$. The total order of the singularities is $-6$ (just because the total discrete curvature is $2$) so at least three conical singularities are double poles of $(df)^{\otimes 3}$ (their conical angle is equal to $\frac{2\pi}{3}$). In particular, $(df)^{\otimes 3}$ is a \textit{primitive} cubic differential (it is not the global power of a meromorphic $1$-form).
\par
Conversely, any primitive cubic differential $\omega$ with zeroes and poles of even order $k \geq -2$ on $\mathbb{CP}^{1}$ defines a flat metric with conical singularities at the zeroes and poles. All the conical angles belong to $\frac{2\pi}{3} + \frac{4\pi}{3}\mathbb{N}$ so the discrete curvatures are integer multiples of $\frac{2}{3}$ and we obtain $\delta \geq \frac{1}{3}$. Since the total order of the singularities of a cubic differential on $\mathbb{CP}^{1}$ is $-6$, $\omega$ has at least three double poles. Each of them is a conical singularity of angle $\frac{2\pi}{3}$. Their discrete curvature is $\frac{2}{3}$ and therefore $\delta=\frac{1}{3}$. In other words, the equality case $\delta=\frac{1}{3}$ exactly coincides with singular flat metrics induced by primitive cubic differentials with singularities of even order on $\mathbb{CP}^{1}$.
\end{proof}

\section{Flat annuli}\label{sec:annuli}

Just like the geometry of a translation surface is largely determined by shapes of its flat cylinders, \textit{flat annuli} play an analogous role for flat surfaces.

\begin{defn}
For any angle $\theta \in ]0,2\pi]$ and radii $0 \leq R < R'$, we define a \textit{flat annulus} $\mathcal{A}_{R,R',\theta}$ by identifying the two radial sides of domain $\lbrace{ z \in \mathbb{C}~\vert~R \le |z| \le R'~;~0\le arg(z)\le \theta \rbrace}$ by a rotation of the angle $\theta$.
\par
We refer to $\theta$ as the \textit{apex angle} of the annulus. The definition extends for any $\theta>0$ by a covering construction.
\end{defn}

The exponential map provides a holomorphic parametrization of the flat annulus $\mathcal{A}_{R,R',\theta}$ by the flat cylinder obtained from a rectangle $\lbrace{ z \in \mathbb{C}~\vert~0 \leq Re(z) \leq ln(R'/R)~;~0 \leq Im(z) \leq \theta \rbrace}$. The \textit{conformal modulus} of the flat annulus $\mathcal{A}_{R,R',\theta}$ is thus $\frac{ln(R'/R)}{\theta}$.
\par
Identifying the center of the flat annulus with $0 \in \mathbb{C}$, the polar coordinate of the complex plane provides a pair of orthogonal foliations on $\mathcal{A}_{R,R',\theta}$:
\begin{itemize}
    \item the \textit{radial foliation}, whose leaves are trajectories contained in lines incident to the center of the annulus;
    \item the \textit{orthoradial foliation}, whose leaves are circular arcs orthogonal to radial leaves.
\end{itemize}
In particular, the curvature radius of orthoradial leaves varies between $R$ and $R'$. We will refer to boundaries of orthoradial leaves of radii $R$ and $R'$ respectively as the \textit{inner boundary arc} and the \textit{outer boundary arc}. If $R=0$, the inner boundary arc is reduced to a conical singularity of the angle $\theta$.

\subsection{Trajectories in a flat annulus}\label{sub:trajannulus}

In any annulus $\mathcal{A}_{R,R',\theta}$, there is an important difference between trajectories starting from the inner boundary arc and trajectories starting from the outer boundary arc. We obtain this description by looking at the lift of the trajectory in the universal cover of the flat annulus.
\par
In any trajectory $t$ starting from the inner boundary arc, the radius (determining a leaf in the orthogonal foliation) along the trajectory is strictly increasing from $R$ to $R'$. It follows that it is simple and leaves the annulus through the outer boundary arc.
\par
For a trajectory $t$ starting from the outer boundary arc, we denote by $A$ the starting point of $t$ and $\alpha$ the angle of $t$ with the radial leaf incident to $A$. There are two different regimes (see Figure~\ref{fig:traj}):

    \begin{figure}[!htbp]
	\centering
	\begin{minipage}{0.5\linewidth}
	\includegraphics[width=\linewidth]{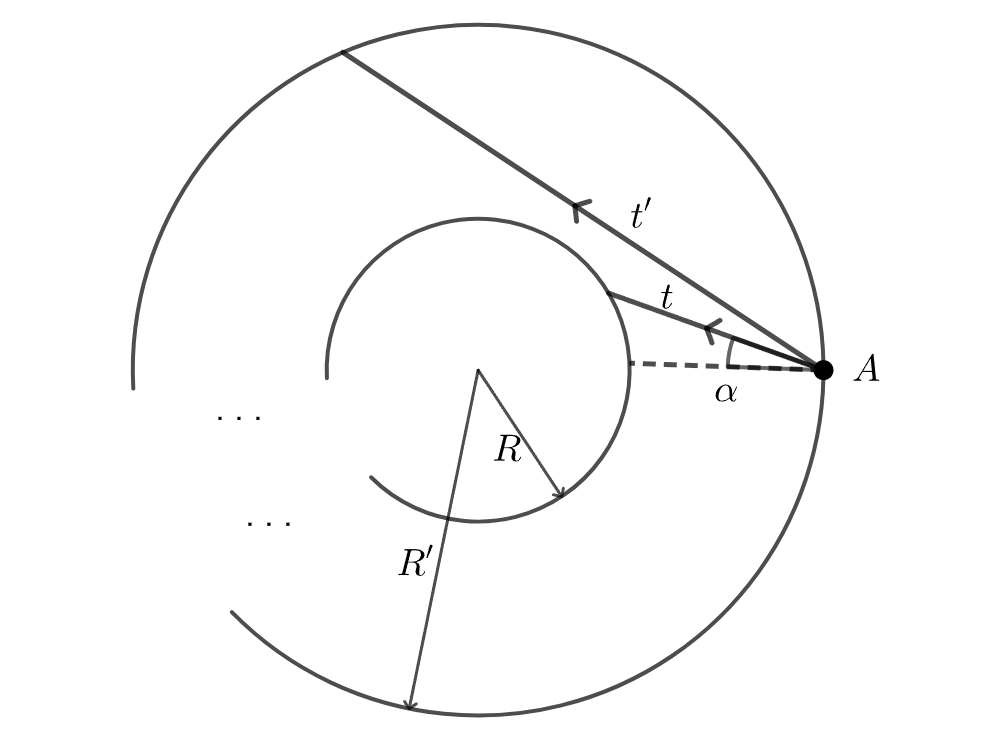}
	\footnotesize
	\end{minipage}
    \caption{Two trajectories $t$ and $t'$ starting at $A$ in the universal cover of a flat annulus.}\label{fig:traj}
    \end{figure}

\begin{itemize}
    \item if $\alpha \leq \arcsin(R'/R)$, the radius is strictly decreasing from $R'$ to $R$ and trajectory $t$ eventually leaves the annulus through the inner boundary arc. The trajectory is automatically simple.
    \item if $\alpha> \arcsin(R'/R)$, the radius is strictly decreasing from $R'$ to $R'\sin(\alpha)$ and then strictly increasing from $R'\sin(\alpha)$ to $R'$. The (possibly self-intersecting) trajectory eventually leaves the annulus through the outer boundary.
\end{itemize}
In the second case, the trajectory is a chord of the outer circle whose central angle is $\pi-2\alpha$. We observe that all the self-intersections of $t'$ belong to two radial leaves (see Figure~\ref{fig:crossing}):
\begin{enumerate}
    \item the radial leaf $\mathcal{L}$ containing the unique point where $t'$ is tangent to the orthoradial foliation and realizes the minimal radius $R'\sin(\alpha)$;
    \item the radial leaf obtained from $\mathcal{L}$ by a rotation of angle $\frac{\theta}{2}$.
\end{enumerate}
A direct counting then proves that the number of self-intersections in trajectory $t'$ is $\lfloor \frac{\pi-2\alpha}{\theta}\rfloor$.

   \begin{figure}[!htbp]
	\centering
	\begin{minipage}{0.5\linewidth}
	\includegraphics[width=\linewidth]{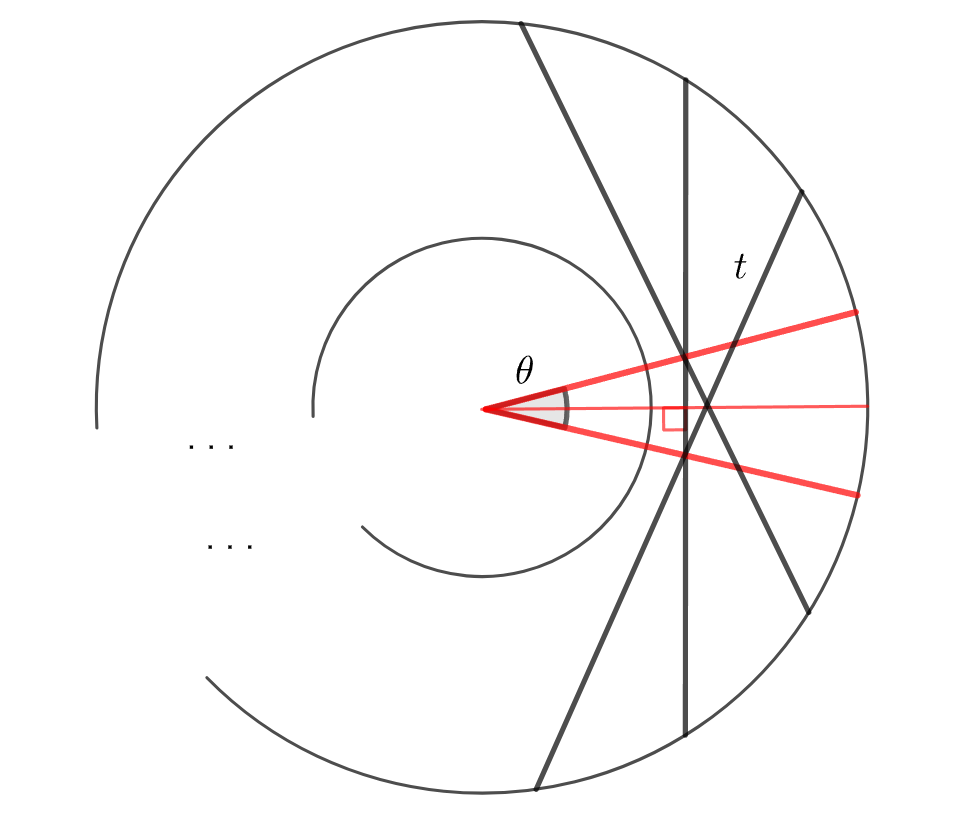}
	\footnotesize
	\end{minipage}
    \caption{The red radial segments are the copies of the leaves that the self-intersections of $t$ belong to.}\label{fig:crossing}
    \end{figure}

\subsection{Monogonal trajectories}\label{sub:monogon}

In a flat sphere, a \textit{monogonal trajectory} is a simple trajectory with the same endpoint and avoids any conical singularity. We call the endpoint the \textit{turning point}. Notice that the trajectory divides the angle at the turning point into two sub-angles. We call the smaller sub-angle the \textit{interior angle}. 
\par
The curvature gap provides constraints on interior angles of monogonal trajectories.

\begin{lem}\label{lem:monogon}
In a flat sphere $X$ with a curvature gap $\delta$, for any monogonal trajectory $\gamma$, the interior angle $\alpha$ of loop $\gamma$ satisfies $\alpha \leq \pi-2\pi\delta$.
\end{lem}

\begin{proof}
Any closed topological loop $\gamma$ decomposes $X$ into connected components and the set of conical singularities into two subsets. If $\gamma$ is a polygonal path with an interior angle, Gauss-Bonnet formula implies that $2\pi\sum k_i + \pi - \alpha = 2\pi$, where the sum runs over the curvatures of the singularities enclosed by $\gamma$. By the definition of the curvature gap, our inequality follows.
\end{proof}

A monogonal trajectory $\gamma$ is characterized by a pair $(\alpha,L)$ where $\alpha \in ]0,\pi[$ is the interior angle while $L$ is the length of the segment.
\par
The holonomy of the flat metric along the loop is a rotation of angle $\pi -\alpha$. Each point of $\gamma$ has a well-defined radius with respect to the fictive center of this rotation. A direct trigonometric computation proves that the maximal radius $\frac{L}{2\cos(\alpha/2)}$ is realized at the turning point while the minimal radius $\frac{L'\tan(\alpha/2)}{2}$ is realized at the midpoint of the unique segment of $\gamma$.
\par

Each monogonal trajectory $\gamma$ can be embedded in a unique way in a $1$-parameter family of disjoint monogonal trajectories $(\gamma_{l})_{L < l < L'}$ such that:
\begin{itemize}
    \item the interior angle of each trajectory of the family is the same;
    \item turning points of monogonal trajectories describe a straight line $\mathcal{L}$ in the flat surface as $l$ changes;
    \item $\mathcal{L}$ is the bisector of the interior angle for each trajectory $\gamma_{l}$.
\end{itemize}
Provided $\frac{L'}{L}> \frac{1}{\sin(\alpha/2)}$, notice that the flat surface which is swept by a family $(\gamma_{l})_{L < l < L'}$
contains an annulus $\mathcal{A}_{\theta,R,R'}$ (see Figure~\ref{fig:monogonal}) where:
\begin{itemize}
    \item $\theta=\pi-\alpha$;
    \item $R= \frac{L}{2\cos(\alpha/2)}$;
    \item $R'=\frac{L'\tan(\alpha/2)}{2}$.
\end{itemize}

\begin{figure}[!htbp]
	\centering
	\begin{minipage}{0.6\linewidth}
	\includegraphics[width=\linewidth]{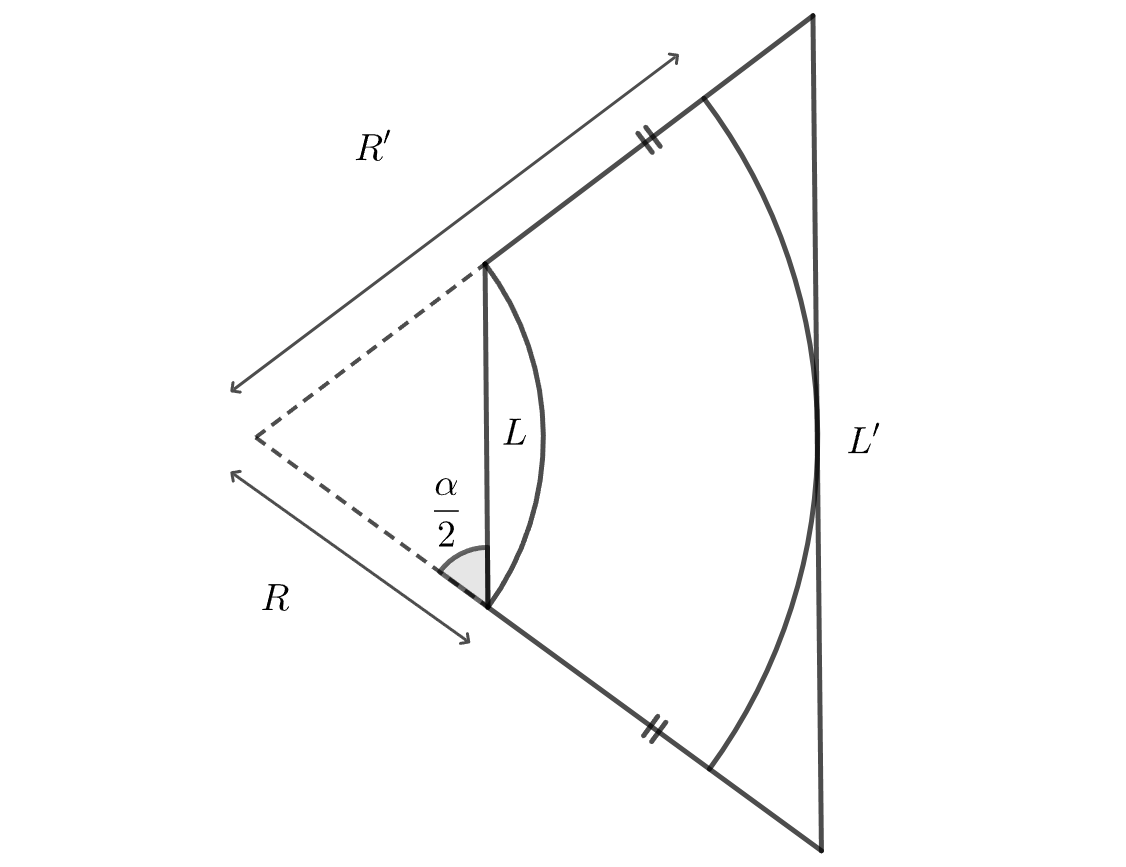}
	\footnotesize
	\end{minipage}
    \caption{The trapezoid in the figure defines a family of monogonal trajectories by identifying its legs. This family contains an annulus.}\label{fig:monogonal}
    \end{figure}

\subsection{Counting saddle connections in a flat annulus}\label{sub:example}

This section is dedicated to explaining why the curvature gap is necessary for our estimates in Theorem~\ref{thm:MAIN} and Theorem~\ref{thm:MAIN2}.
\par
For $0 \leq R < R'$ and $0 < \theta< \pi$, we consider a flat annulus $\mathcal{A}_{R,R',\theta}$ embedded in a flat sphere such that each boundary arc contains exactly one singularity. We will also assume that these singularities $A,B$ belong to the same radial leaf.

\begin{prop}\label{prop:almost}
Flat annulus $\mathcal{A}_{R,R',\theta}$ contains at least $\arccos{(\frac{R}{R'})}\cdot\frac{1}{\theta}$ simple saddle connections.
\end{prop}

For a flat sphere $X$ containing a flat annulus $\mathcal{A}_{R,R',\theta}$, Lemma~\ref{lem:monogon} proves that apex angle $\theta$ satisfies a lower bound $\theta \geq 2\pi\delta$ where $\delta$ is the curvature gap $\delta$ of $X$. 
When the apex angle $\theta$ goes to zero, the curvature gap of the surface goes to zero and the number and the length of simple saddle connections tend to infinity. This indicates why the curvature gap $\delta$ is necessary in the upper bound of Theorem~\ref{thm:MAIN}.
\par
Taking the limit $\theta \rightarrow 0$ while bounding the conformal modulus of the flat annulus, we obtain in the limit a flat cylinder containing infinitely many saddle connections with arbitrarily large lengths. In a flat sphere, this can only happen if $\delta=0$.

\begin{proof}[Proof of Proposition~\ref{prop:almost}]
We develop counterclockwise the flat annulus and denote by $A_{0}, A_{1},\dots,A_{k}$ the copies of $A$ as in Figure~\ref{fig:unfolding}.
\par
We define $I$ as the tangency point for a line $IB$ incident to $B$ and tangent to the circle of radius $R$. Defining $\beta$ as $\angle IOB$, $I$ belongs to the circular arc between two successive copies $A_{m}$ and $A_{m+1}$ of $A$, see Figure~\ref{fig:unfolding}.
\par
Since $\cos(\beta)=\frac{R}{R'}$, we have $m\theta \leq \arccos(\frac{R}{R'}) \leq (m+1)\theta$. It follows that segments $[BA_{0}],\dots,[BA_{m}]$ form a family of $\arccos(\frac{R}{R'})\cdot\frac{1}{\theta}$ simple saddle connections.

    \begin{figure}[!htbp]
	\centering
	\begin{minipage}{1\linewidth}
	\includegraphics[width=\linewidth]{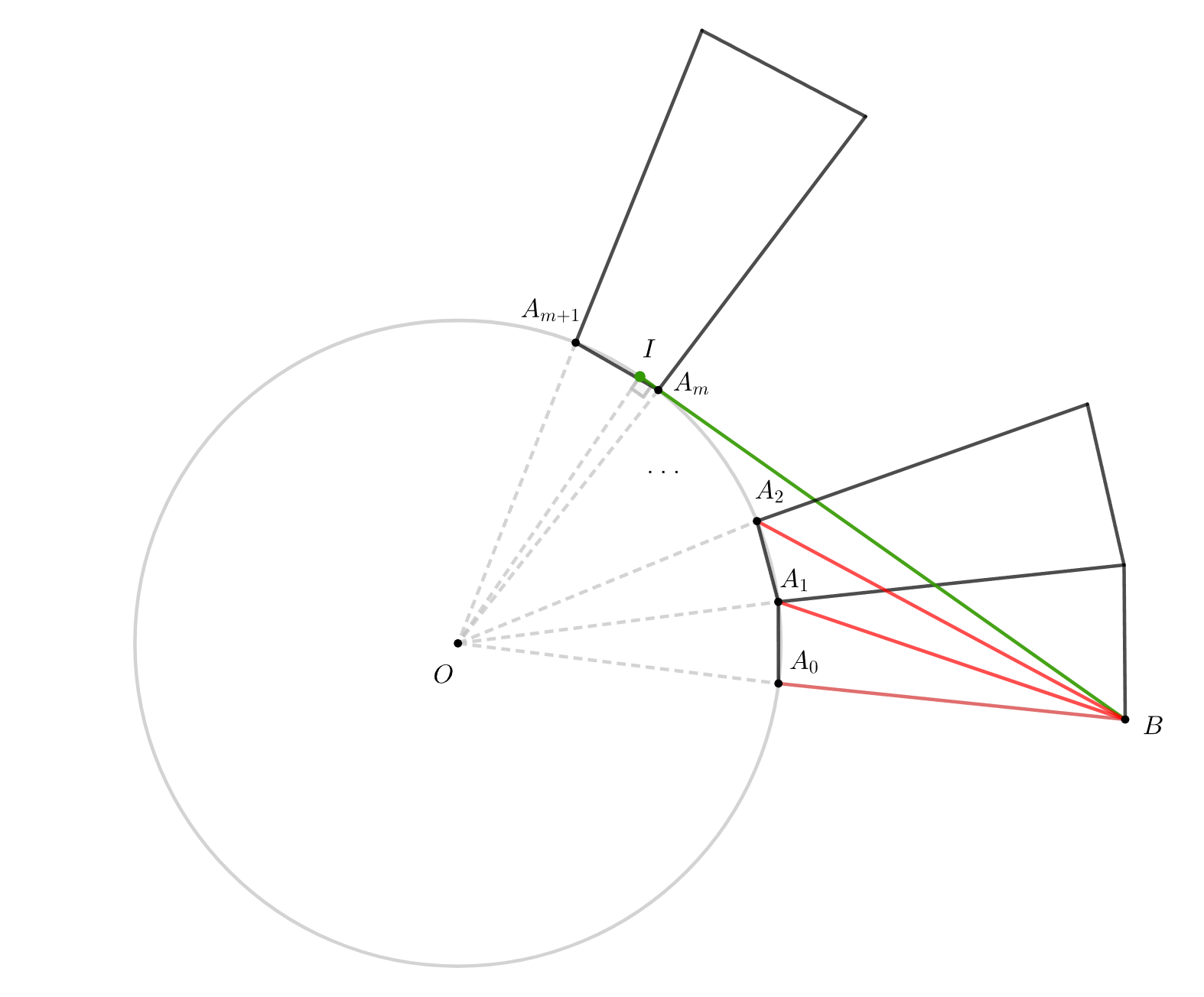}
        \caption{An unfolding of the almost cylinder, together with the sequence of saddle connections $[BA_0], \dots, [BA_m]$. The segment $IB$ is tangent to the circle at the point $I$.}\label{fig:unfolding}
	\footnotesize
	\end{minipage}
    \end{figure}
\end{proof}

\section{Geometry of immersed disks}\label{sec:Delaunay}

\subsection{Locally isometric immersions of disks}\label{sub:isoimmersion}

In the study of a flat surface, a crucial geometric quantity is the size of the largest open disk we can immerse or embed in the surface without hitting the conical singularities.

\begin{prop}\label{prop:immersion}
If there is a locally isometric immersion $f$ of an open disk of radius $\lambda >\frac{1}{\sqrt{\pi}}$ in a flat sphere $X$ of unit area, then $X$ contains an embedded flat annulus $\mathcal{A}_{R,R',\theta}$ where:
\begin{itemize}
    \item the apex angle $\theta$ satisfies $\theta < \frac{\pi}{4\pi\lambda^{2}-4}$;
    \item the conformal modulus $\mu$ satisfies $\mu > \sqrt{\pi\lambda^{2}-1}$.
\end{itemize}
The degenerate case of a translation cylinder corresponds to a flat annulus with angle $\theta=0$.
\end{prop}

\begin{proof}
We identify the domain of immersion $f$ with the centered disk $\mathcal{D}$ of radius $\lambda$ in the complex plane and set $M=f(0)$. Since $f$ is locally injective, $M$ is the center of a maximal embedded disk $\mathcal{D}^{o}$ in $X$. The area of $\mathcal{D}^{o}$ is bounded by the area of $X$ so the radius $\rho$ of $\mathcal{D}^{o}$ satisfies $\rho < \frac{1}{\sqrt{\pi}}$. In particular, we have $\lambda>\rho$ and $\mathcal{D}^{o}$ is contained in $\mathcal{D}$.
\par
We deduce from the maximality hypothesis on $\mathcal{D}^{o}$ that two points $A,B$ of $\partial \mathcal{D}^{o}$ have the same image under the immersion. Without loss of generality, we suppose that the mediatrix of the segment $[A,B]$ contains the horizontal diameter of disk $\mathcal{D}$. We locate points $A,B$ with corresponding complex numbers $\rho e^{i\alpha}$ and $\rho e^{-i\alpha}$ respectively for $\alpha \in ]0,\frac{\pi}{2}]$ see Figure~\ref{fig:Immersion}. We first assume that $\alpha \in ]0,\frac{\pi}{2}[$
\par
\begin{figure}[!htbp]
	\centering
	\begin{minipage}{0.75\linewidth}
	\includegraphics[width=\linewidth]{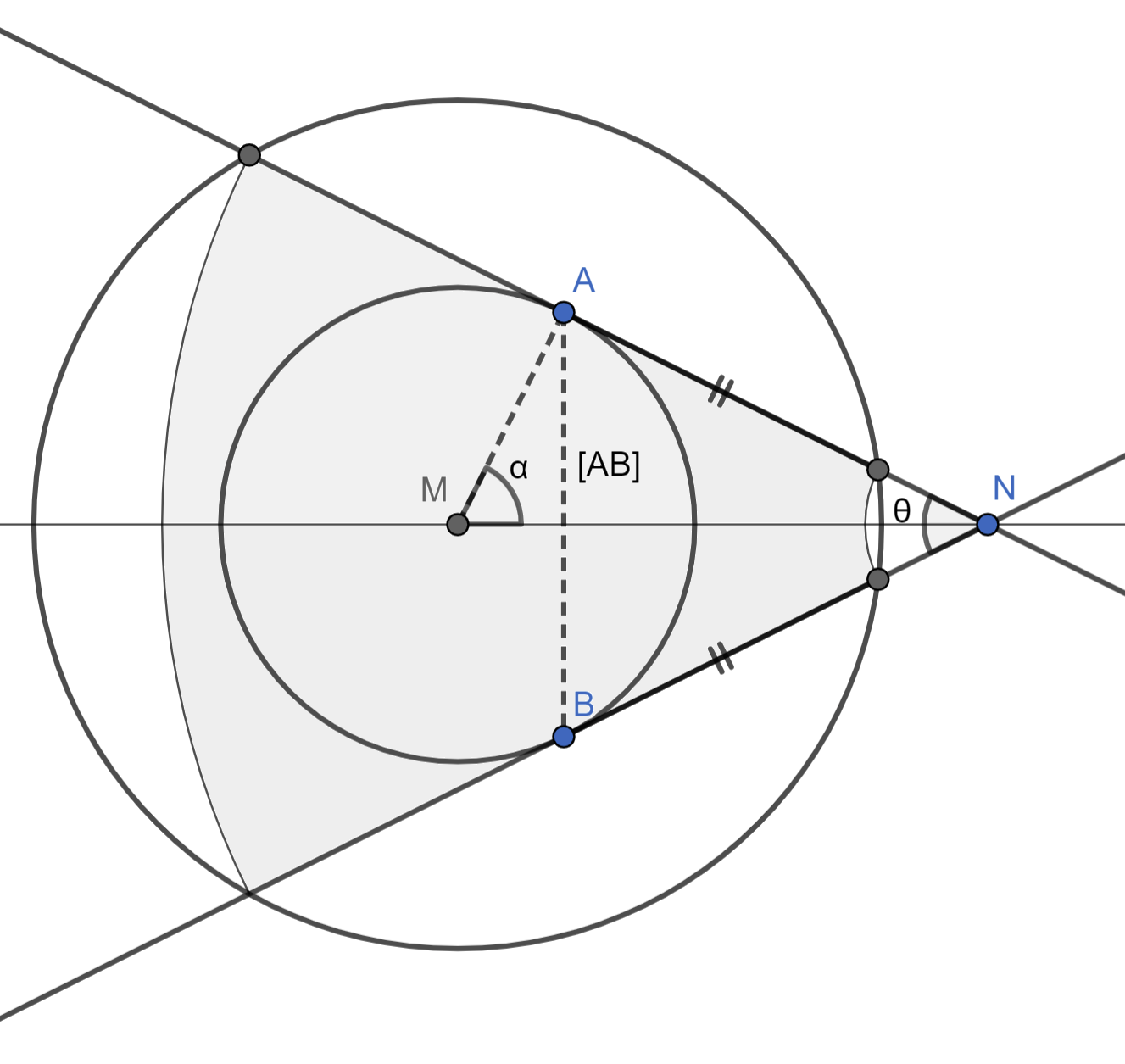}
        \caption{Delaunay disk $\mathcal{D}$, embedded disk $\mathcal{D}^{o}$, monogonal trajectory $[AB]$ and the flat annulus in grey.}\label{fig:Immersion}
	\footnotesize
	\end{minipage}
\end{figure}
A neighborhood of $B$ in the tangent line $T_{B}$ of $\partial \mathcal{D}^{o}$ of $B$ has the image under $f$ as some segment containing $A$. Since $\mathcal{D}^{o}$ is the maximal embedded disk centered at $M$, this segment has also to be contained in the tangent line $T_{A}$ to $\partial \mathcal{D}^{o}$. It follows that segment $[AB]$ belongs to a family of vertical monogonal trajectories (see Section~\ref{sub:monogon}) with turning points located at these identified tangent lines.
\par
This family of curves can be continued on $X$ until reaching a conical singularity. Since $f(\mathcal{D})$ cannot contain any conical singularity, it follows that the intersection point $N$ of these two tangent lines is outside $\mathcal{D}$, see Figure~\ref{fig:Immersion}. Identify point $N$ with the real number $x$ of the complex plane, we have $x \geq \lambda$.
\par
Each of the two identified points between tangent lines $T_{A},T_{B}$ is connected by a circular arc of the center $N$ contained in disk $\mathcal{D}$. We denote by $R<R'$ the radii of these arcs. The images under $f$ of these two circular arcs bound a flat annulus $\mathcal{A}_{R,R',\theta}$ where the apex angle $\theta$ satisfies $\theta = \pi -2\alpha$.
\par
The area $A=\theta(R'+R)(R'-R)$ of this annulus is bounded by the total area of the surface. Since we have $R'-R=2\sqrt{\lambda^{2}-\rho^{2}}$ and $\rho^{2}>\frac{1}{\pi}$, we deduce the inequality $A>4\theta(\lambda^{2}-\frac{1}{\pi})$. It follows from $A<1$ that $\theta < \frac{\pi}{4\pi\lambda^{2}-4}$.
\par
It remains to give a lower bound on the conformal modulus of $\mathcal{A}_{R,R',\theta}$. In fact, the modulus is $\frac{ln(R'/R)}{\theta}$. We know $\theta=\pi-2\alpha$. Since $R'-R=2\sqrt{\lambda^{2}-\rho^{2}}$ and $R'+ R=2\rho\tan{\alpha}$, we have $R'=\rho\tan{\alpha} + \sqrt{\lambda^2-\rho^2}$ and $R=\rho\tan{\alpha} - \sqrt{\lambda^2-\rho^2}$. Hence, the conformal modulus $\mu$ of $\mathcal{A}_{R,R',\theta}$ is given by:

$$\mu = \frac{\ln\left({\frac{\rho\tan{\alpha} + \sqrt{\lambda^2-\rho^2}}{\rho\tan{\alpha} - \sqrt{\lambda^2-\rho^2}}}\right)}{\pi-2\alpha}.$$

According to the inequality $\ln(1+x)\ge\frac{x}{1+x}$ for $x>-1$, we know that:

$$\frac{\ln\left({\frac{\rho\tan{\alpha} + \sqrt{\lambda^2-\rho^2}}{\rho\tan{\alpha} - \sqrt{\lambda^2-\rho^2}}}\right)}{\pi-2\alpha}\ge \frac{2\sqrt{\lambda^2-\rho^2}}{\rho\tan{\alpha}-\sqrt{\lambda^2-\rho^2}}\cdot \frac{1}{\pi-2\alpha}.$$

As $\alpha$ increases, the right side of the inequality decreases. In the limit $\alpha \to \frac{\pi}{2}$, the flat annulus becomes a flat cylinder and we obtain:
$$\frac{2\sqrt{\lambda^2-\rho^2}}{\rho\tan{\alpha}-\sqrt{\lambda^2-\rho^2}} \cdot \frac{1}{\pi-2\alpha} \geq \frac{\sqrt{\lambda^2-\rho^2}}{\rho}.$$
It follows then from inequality $\rho<\frac{1}{\sqrt{\pi}}$ that the conformal modulus $\mu$ of $\mathcal{A}_{R,R',\theta}$ is strictly larger than $\sqrt{\pi\lambda^{2}-1}$.
\par
In the limit case where $\alpha=\frac{\pi}{2}$, $T_{A}$ and $T_{B}$ intersect at infinity. The segment $[AB]$ is a simple closed geodesic that belongs a family of vertical closed geodesics forming a translation cylinder which identifies with a flat annulus with an apex angle $\theta=0$. The family of vertical closed geodesics sweeps out a rectangle of width $2\sqrt{\lambda^2-\rho^2}$ and height $2\rho$ so the conformal modulus of the translation cylinder is larger than $\frac{\sqrt{\lambda^2-\rho^2}}{\rho}$. Inequality $\mu > \sqrt{\pi\lambda^{2}-1}$ also holds in that case.
\end{proof}

\begin{rmk}
An interesting geometric result due to Masur and Smillie (Corollary 5.5 in \cite{MS91}) proves that a translation surface of large diameter contains a long flat cylinder. Proposition~\ref{prop:immersion} provides the analog result for flat spheres (and in fact flat surfaces of arbitrary genus since the proof relies on no topological assumption), see \cite{Thu} for discussion of flat spheres near the boundary of the moduli space.    
\end{rmk}

\subsection{Delaunay triangulations}\label{sub:Delaunay}

In any flat sphere $X$ (and more generally any flat surface with conical singularities), the \textit{Voronoi cell} of a singularity $M$ is the polygonal domain in $X$ consisting of points that are closer (with respect to the flat metric) to $M$ than to any other conical singularity of $X$. This defines the \textit{Voronoi tessellation} of $X$.
\par
The \textit{Delaunay polygonation} is a construction dual to the Voronoi tessellation. We consider locally isometric immersions of disks into $X$. When at least three points of the boundary circle are mapped to conical singularities of $X$, these disks are called \textit{Delaunay disks} and the immersion of the convex hull of these singularities into $X$ is called a \textit{Delaunay polygon}. They form a decomposition of $X$ dual to the Voronoi tessellation, see Proposition~3.1 of \cite{Thu} for details on these constructions.
\par
If more than three conical singularities of $X$ are cocyclic, some Delaunay polygons may be convex polygons with more than three sides. A \textit{Delaunay triangulation} is any subdivision of the Delaunay polygonation into triangles. We refer to edges of such a triangulation as \textit{Delaunay edges}.
\par
The following lemma uses the curvature gap to bound the length of Delaunay edges on any flat sphere.

\begin{lem}\label{lem:Delaunay}
In a flat sphere $X$ of unit area with $n$ conical singularities and a curvature gap $\delta>0$, the length $L$ of any edge of a Delaunay triangulation of $X$ satisfies $L^{2} < \frac{4}{\pi}+\frac{1}{2\pi\delta}$.
\end{lem}

\begin{proof}
Any Delaunay edge of length $L$ is contained in the image of a Delaunay disk $\mathcal{D}$ of radius $\lambda \geq \frac{L}{2}$ by a locally isometric immersion $f$. If $\lambda \leq \frac{1}{\sqrt{\pi}}$, then $L^{2} \leq \frac{4}{\pi} < \frac{4}{\pi}+\frac{1}{2\pi\delta}$.
\par
If, instead, we have $\lambda > \frac{1}{\sqrt{\pi}}$, it follows that from Proposition~\ref{prop:immersion} that the image of $\mathcal{D}$ contains a flat annulus of apex angle $\theta < \frac{\pi}{4\pi\lambda^{2}-4}$. Orthoradial leaves of this flat annulus (see Section~\ref{sec:annuli}) decompose $X$ into two connected components whose total curvatures are respectively $\frac{2\pi-\theta}{2\pi}$ and $\frac{2\pi+\theta}{2\pi}$. It follows then from the definition of the curvature gap that we have $\theta \geq 2\pi\delta$ and thus we get $\lambda^{2} < \frac{1}{\pi} + \frac{1}{8\pi\delta}$. The bound on $L$ follows.
\end{proof}

\subsection{Configurations of chords in Delaunay disks}\label{sub:Configurations}

Locally isometric immersions of disks provide also our main tool to control the complexity of simple trajectories in a flat sphere.
\par
Consider a Delaunay triangle $\mathcal{T}$ in a flat sphere $X$, together with the locally isometric immersion $f: \mathcal{D} \rightarrow X$ of the corresponding Delaunay disk $\mathcal{D}$. Note that for a simple saddle connection $t$, its endpoints do not belong to the image of the interior of a Delaunay disk by a locally isometric immersion $f$. It follows that $f^{-1}(t)$ is formed by chords of the disk $\mathcal{D}$. Since $t$ is simple, these chords do not intersect each other. We study the configuration of these disjoint chords in $\mathcal{D}$.

Let $c_1$ and $c_2$ be two disjoint chords in a disk $\mathcal{D}$. We denote by $I(c_1,c_2)$ the closed domain in $\mathcal{D}$ bounded by $c_1$ and $c_2$. 
We define
$$
\beta(c_1,c_2) : = \frac{|\partial \mathcal{D} \cap I|}{r},
$$
where $|\partial \mathcal{D} \cap I|$ denotes the arc-length and $r$ denotes the radius of $\mathcal{D}$. In other words, $\beta(c_1,c_2)$ is the sum of the central angles corresponding to the arcs in $\partial \mathcal{D} \cap I$.

We say that $c_1$ and $c_2$ are \emph{consecutive in $\mathcal{T}$} if they are the only chords of $f^{-1}(t)$ contained in the domain $I(c_1, c_2)$.

\begin{lem}\label{lem:chords}
In a flat sphere $X$ with positive curvature gap $\delta>0$, we consider a Delaunay triangle $\mathcal{T}$ and the locally isometric immersion $f: \mathcal{D} \rightarrow X$ of its Delaunay disk. Let $t$ be a simple saddle connection, and let $c_1, c_2, c_3 $ be three chords in $f^{-1}(t)$ that intersect $\mathcal{T}$. The parametrization of $t$ induces orientations on $c_i$ for each $i$. Then:
\begin{enumerate}
    \item If $c_1$ and $c_2$ are consecutive in $\mathcal{T}$ with the same orientation, then $\beta(c_1, c_2) \geq 4\pi\delta$.
    \item If $(c_1,c_2)$ and $(c_2,c_3)$ are two consecutive pairs of chords, then $\beta(c_1,c_3) = \beta(c_1, c_2) + \beta(c_2, c_3) \geq 4\pi\delta$.
\end{enumerate}
\end{lem}

\begin{proof}
We define $\theta_i$ as the angle obtained by rotating the rightward-pointing horizontal unit vector counterclockwise to the tangent vector of $c_i$. Thus, $\theta_i$ lies in the interval $[0, 2\pi)$. For each $c_i$, we choose a point $A_i\in c_i\cap\mathcal{T}$.

\noindent\underline{Proof of~(1):}

There exists a permutation $\sigma \in \mathfrak{S}_2$ such that the order of the chords $(c_{\sigma(1)}, c_{\sigma(2)})$ is induced by the parametrization of $t$. Up to reversing the orientation of $t$, we may assume that $\sigma = \mathrm{id}$.

Inside $\mathcal{T}$, we draw an oriented segment $\gamma$ from $A_{2}$ to $A_{1}$. 
Since $I(c_1,c_2)$ contains no other chords, the segment $\gamma$ and the portion of the simple trajectory $t$ from $A_1$ to $A_2$ form a closed polygonal loop with no self-intersection on $X$. This loop divides $X$ into two domains. Denote by $P$ the domain whose interiors angles are $\alpha_1$ and $\alpha_2$, where $\alpha_i$ is defined as in Figure~\ref{fig:twochordsameorientation}. Applying the Gauss-Bonnet formula to the domain $P$, we obtain that
\begin{equation}\label{equ:case1gb}
    2\pi k + \pi-\alpha_1 + \pi-\alpha_2= 2\pi,
\end{equation}
where $k$ is the sum of the discrete curvatures in $P$. 

\begin{figure}[!htbp]
	\centering
	\begin{minipage}{\linewidth}
	\includegraphics[width=\linewidth]{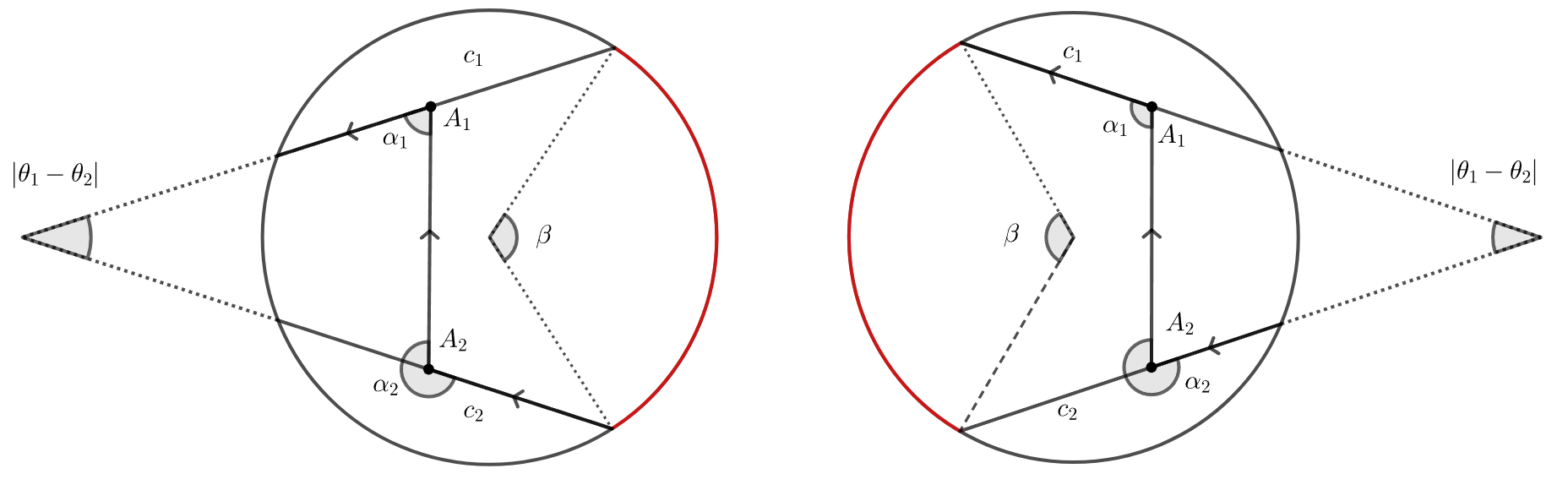}
	\footnotesize
	\end{minipage}
    \caption{Delaunay disk $\mathcal{D}$, two cases of two consecutive chords $c_{1},c_{2}$ with same orientations.}\label{fig:twochordsameorientation}
\end{figure}

As shown in Figure~\ref{fig:twochordsameorientation}, there are two possible cases.
\begin{itemize}
    \item If $|\theta_1-\theta_2|\leq \pi$ (the left picture), then the identity~\eqref{equ:case1gb} simplifies to $|\theta_1 - \theta_2| = 2\pi (1-k) \geq 2\pi\delta$. By the inscribed angle theorem, the angle $\beta$ in the picture is at least $4\pi\delta$. It follows that $\beta(c_1, c_2) \geq 4\pi\delta$.
    \item If $|\theta_1-\theta_2|> \pi$ (the right picture), then identity~\eqref{equ:case1gb} simplifies to $2\pi - |\theta_1 - \theta_2| = 2\pi (k - 1) \geq 2\pi\delta$. Again, by the inscribed angle theorem, $\beta(c_1, c_2) \geq \beta \geq 4\pi\delta$.
\end{itemize}

\noindent\underline{Proof of~(2):}
Before proving the case of three consecutive chords, we first analyze a complementary case to case~(1).\newline
\noindent\textbf{Step~1:} Consecutive $c_1$ and $c_2$ but with different orientations.

As in the proof above, we may assume that the trajectory $t$ passes through $c_1$ before $c_2$.

Inside $\mathcal{T}$, we draw an oriented segment $\gamma$ from $A_{2}$ to $A_{1}$. The segment $\gamma$ and the portion of the simple trajectory $t$ from $A_1$ to $A_2$ form a closed polygonal loop with no self-intersection on $X$. Denote by $P$ the domain whose interior angles $\alpha_1$ and $\alpha_2$ are strictly less than $\pi$, where $\alpha_i$ is defined as in Figure~\ref{fig:twochordoppositeorientation}. Applying the Gauss-Bonnet formula to the domain $P$, we obtain that
\begin{equation}\label{equ:case2gb}
    2\pi k + \pi-\alpha_1 + \pi-\alpha_2 = 2\pi,
\end{equation}
where $k$ is the sum of the discrete curvatures in $P$. 

\begin{figure}[!htbp]
	\centering
	\begin{minipage}{\linewidth}
	\includegraphics[width=\linewidth]{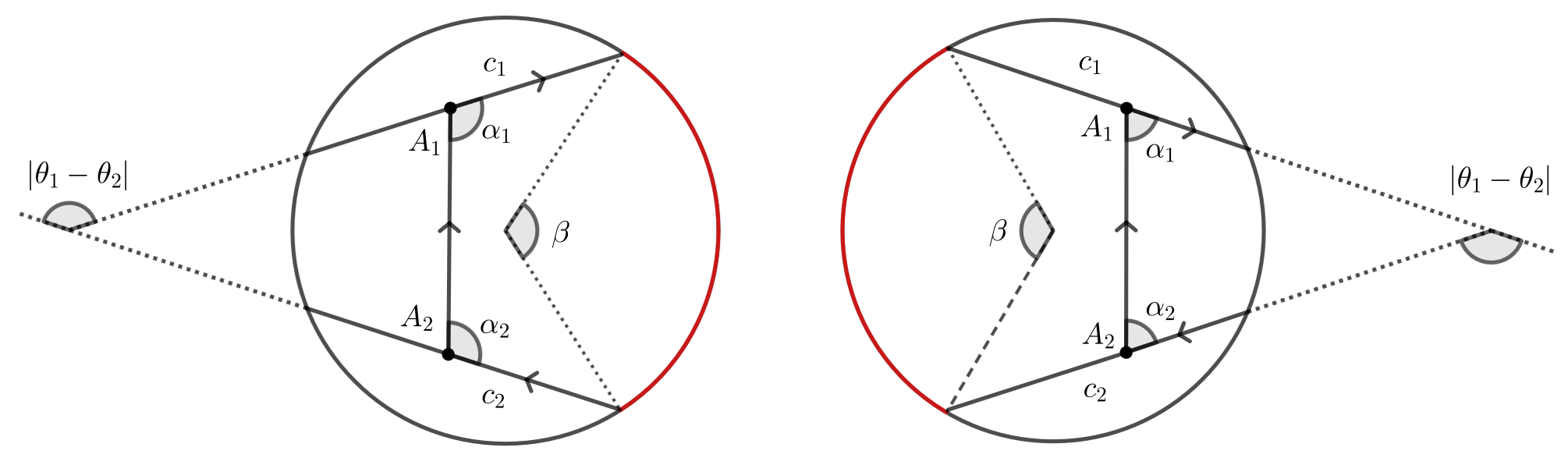}
	\footnotesize
	\end{minipage}
    \caption{Delaunay disk $\mathcal{D}$, two cases of two consecutive chords $c_{1},c_{2}$ with different orientations.}\label{fig:twochordoppositeorientation}
\end{figure}

As shown in Figure~\ref{fig:twochordoppositeorientation}, there are two possible cases.
\begin{itemize}
    \item \textbf{Case~1}: If $|\theta_1-\theta_2|\leq \pi$ (the left picture), then the identity~\eqref{equ:case2gb} simplifies to $|\theta_1 - \theta_2| = 2\pi (1-k)$. By the inscribed angle theorem, we have that $\beta \geq 2\big(\pi - |\theta_1 - \theta_2|\big) = 4\pi\Big(k - \frac{1}{2}\Big)$. In particular, since $\pi - |\theta_1 - \theta_2|\geq0$, it follows that $k-\frac{1}{2}\geq0$.
    \item \textbf{Case~2}: If $|\theta_1-\theta_2|> \pi$ (the right picture), then identity~\eqref{equ:case2gb} simplifies to $|\theta_1 - \theta_2| = 2\pi (1 - k)$. In particular, $1-k > 0$. By the inscribed angle theorem, we have that $\beta \geq 2\big(|\theta_1 - \theta_2| - \pi\big) = 4\pi\Big(\frac{1}{2} - k\Big)$. In particular, since $|\theta_1 - \theta_2| - \pi\geq0$, it follows that $\frac{1}{2} - k\geq0$
\end{itemize}
Therefore, in either case, we conclude that
\begin{equation}\label{equ:case22}
    \beta \geq 4\pi \left| \frac{1}{2} - k \right|.
\end{equation}

\noindent\textbf{Step~2:} Consecutive chords $c_1, c_2, c_3$.

According to the conclusion~(1), it remains to consider the case where three consecutive chords have alternating orientations; see Figure~\ref{fig:threeconsecutive}. Let $\sigma \in \mathfrak{S}_3$ be a permutation such that the order of the chords $(c_{\sigma(1)}, c_{\sigma(2)}, c_{\sigma(3)})$ is induced by the parametrization of $t$. Up to reversing the global orientation of trajectory $t$, there are just two cases to consider:
\begin{itemize}
    \item $\sigma(1)<\sigma(3)<\sigma(2)$;
    \item $\sigma(1)<\sigma(2)<\sigma(3)$.
\end{itemize}

In the first case, by discarding the portion of $t$ following the chord $c_3$, we may apply the previously established assertion~(1) to conclude that $\beta(c_1, c_3) \geq 4\pi\delta$.

In the second case, for each $i = 1, 2$, the segment from $A_{i+1}$ to $A_i$, together with the portion of $t$ from $A_{i+1}$ to $A_i$, forms a simple polygonal loop $l_i$ in $X$. These two loops together form a figure-eight-shaped polygonal path on $X$. As in Step~1, let $P_i$ denote the domain enclosed by the loop $l_i$, which has interior angles strictly less than $\pi$. In particular, the interiors of $P_1$ and $P_2$ are disjoint.

Define the central angles $\beta$ and $\beta'$ as in Figure~\ref{fig:threeconsecutive}. By Step~1, we have
$$
\beta + \beta' \geq 4\pi\left|\tfrac{1}{2} - k\right| + 4\pi\left|\tfrac{1}{2} - k'\right| \geq 4\pi\left|1 - k - k'\right|,
$$
where $k$ and $k'$ denote the sums of the discrete curvatures in $P_1$ and $P_2$, respectively. Since $P_1$ and $P_2$ have disjoint interiors, the curvatures in $P_1$ and $P_2$ are at different singularities. The assumption of a curvature gap implies that $|1 - k - k'| \geq \delta$. Therefore, we obtain
$$
\beta(c_1, c_2) + \beta(c_2, c_3) \geq \beta + \beta' \geq 4\pi\delta.
$$
\end{proof}

\begin{figure}[!htbp]
	\centering
	\begin{minipage}{0.6\linewidth}
	\includegraphics[width=\linewidth]{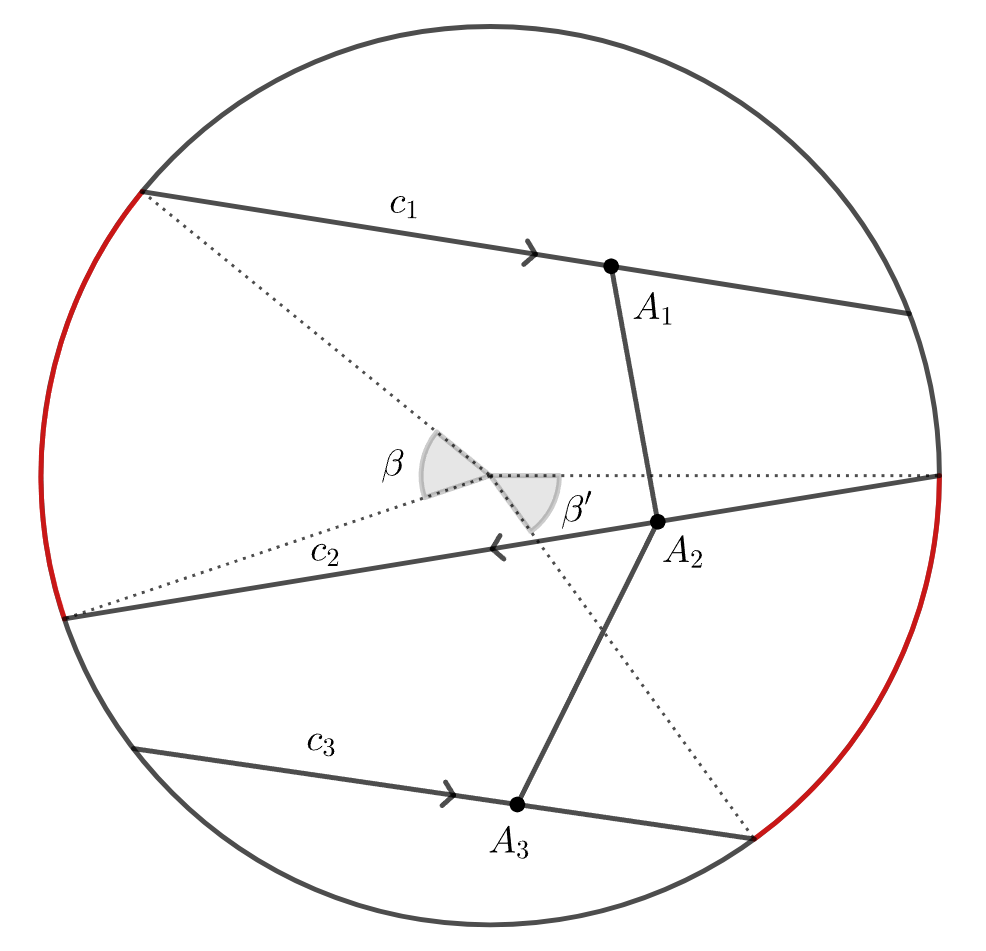}
	\footnotesize
	\end{minipage}
    \caption{Delaunay disk $\mathcal{D}$, three consecutive chords $c_{1},c_{2}, c_3$ with alternative orientations.}\label{fig:threeconsecutive}
\end{figure}

We deduce from Lemma~\ref{lem:chords} an upper bound on the number of chords in the intersection of a simple saddle connection and the Delaunay disk.

\begin{lem}\label{lem:disk}
In a flat sphere $X$, we consider a Delaunay triangle $\mathcal{T}$ and its Delaunay disk locally isometric immersion $f: \mathcal{D} \rightarrow X$.
\par
Assuming that the curvature gap $\delta$ of $X$ satisfies $\delta>0$, for any simple saddle connection $t$, the intersection $f^{-1}(t) \cap \mathcal{T}$ consists of at most $\frac{5}{2\delta}$ chords of disk $\mathcal{D}$.
\end{lem}

\begin{proof}
We denote by $\mathcal{C}h$ the subset of the (disjoint) chords in $f^{-1}(t)$ that cross the Delaunay triangle $\mathcal{T}$. We aim to show that $\mathcal{C}h$ is finite and that its cardinality is bounded above by $\frac{5}{2\delta}$.

\noindent\textbf{Step 1. $\mathcal{C}h$ is finite.}

We assume by contradiction that $f^{-1}(t)$ contains infinitely many chords in $\mathcal{D}$. Then there is a point $x \in X$ (possibly singular) and a radius $r>0$ smaller than the length of any saddle connection of $X$ such that the metric disk of radius $r$ centered on $x$ intersects $t$ in infinitely many disjoint chords accumulating on $x$. As they accumulate on $x$, the lengths of these chords tend to $2r$. This contradicts the fact that the length of $t$ is finite.

To construct an upper bound on the number of chords in $\mathcal{C}h$, we introduce the following combinatorial object.

\noindent\textbf{Step 2. Dual tree $\Gamma$ associated to $\mathcal{C}h$.}

The chords in $\mathcal{C}h$ decompose the triangle $\mathcal{T}$ into several connected components. Denote by $\Gamma$ the dual graph of this decomposition, where each vertex corresponds to a connected component and edges correspond to chords in $\mathcal{C}h$. Note that the number of edges of $\Gamma$ is equal to the cardinality of $\mathcal{C}h$.

Since the chords in $\mathcal{C}h$ are pairwise disjoint within $\mathcal{T}$, the dual graph $\Gamma$ satisfies the following properties:
\begin{itemize}
    \item $\Gamma$ is a tree;
    \item At most one vertex has valency three;
    \item All other vertices have valency at most two.
\end{itemize}
An example of such a tree $\Gamma$ is illustrated in Figure~\ref{fig:dualtree}.

\begin{figure}[!htbp]
	\centering
	\begin{minipage}{0.6\linewidth}
	\includegraphics[width=\linewidth]{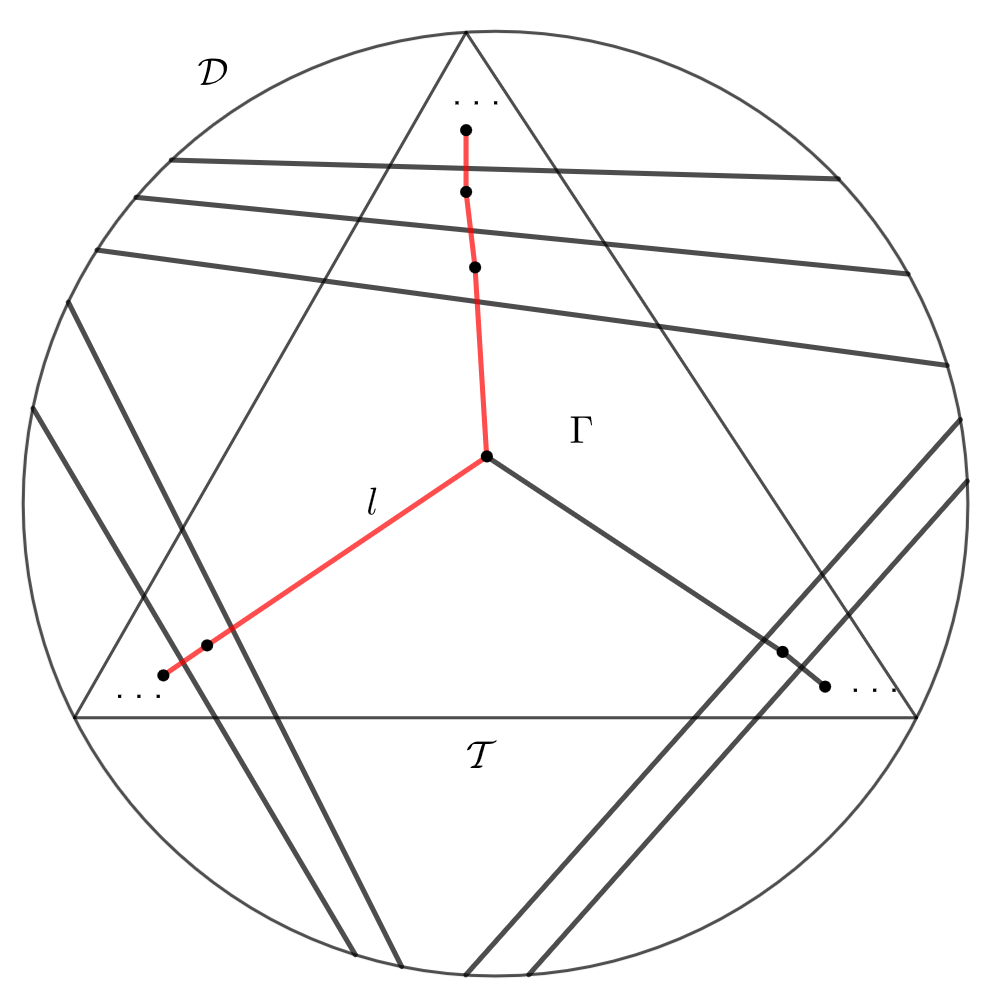}
	\footnotesize
	\end{minipage}
    \caption{The tree $\Gamma$ is the dual graph of the connected components of the triangle $\mathcal{T}$ subdivided by the chords. The red path $l$ is a linear path.}\label{fig:dualtree}
\end{figure}

A \emph{linear path} in $\Gamma$ is a subgraph in which every vertex has valency at most $2$, and which is homeomorphic to a (possibly closed) interval; see Figure~\ref{fig:dualtree}. We say that a linear path is \emph{maximal} if it is not properly contained in any other linear path of $\Gamma$. Since at most one vertex of $\Gamma$ has valency $3$ and all others have valency at most $2$, it follows that $\Gamma$ contains at most three maximal linear paths.

In the next step, we construct an upper bound for the number of edges in a linear path.

\noindent\textbf{Step 3. Lengths of maximal linear paths of $\Gamma$.}

Let $\mathcal{L}$ be a maximal linear path in $\Gamma$. Note that the edges of $\Gamma$ are dual to the chords in $\mathcal{C}h$. Denote by $(c_1, \ldots, c_k)$ the sequence of chords dual to the edges in $\mathcal{L}$, indexed so that $c_i$ and $c_{i+1}$ are consecutive chords in $\mathcal{T}$ (equivalently, dual to adjacent edges in $\mathcal{L}$).
\par
To estimate the number $k$ of edges in $\mathcal{L}$, we consider the sum $\sum\limits_{i=1}^{k-1} \beta(c_i, c_{i+1})$. By Lemma~\ref{lem:disk}, we know that for any three consecutive chords $c_i, c_{i+1}, c_{i+2}$, the following inequality holds:
\begin{equation}\label{equ:linearpath}
    \beta(c_i, c_{i+1}) + \beta(c_{i+1}, c_{i+2}) \ge 4\pi\delta.
\end{equation}

On the other hand, by definition, the sum $\sum\limits_{i=1}^{k-1} \beta(c_i, c_{i+1})$ is equal to the total central angles of a collection of disjoint arcs on the Delaunay disk $\mathcal{D}$. In particular, this sum is bounded above by $2\pi$. By grouping the terms in the sum $\sum\limits_{i=1}^{k-1} \beta(c_i, c_{i+1})$ into pairs, we obtain that when $k$ is even,
$$
2\pi \geq \sum\limits_{i=1}^{k-1} \beta(c_i, c_{i+1}) \geq 4\pi\delta \cdot \frac{k-2}{2},
$$
and when $k$ is odd,
$$
2\pi \geq \sum\limits_{i=1}^{k-1} \beta(c_i, c_{i+1}) \geq 4\pi\delta \cdot \frac{k-1}{2}.
$$
It follows that $k \leq 2 + \frac{1}{\delta}$.

Finally, we are ready to construct the upper bound on the number of chords in $\mathcal{C}h$.

\noindent\textbf{Step 4. Cardinality of $\mathcal{C}h$.}

Estimating the cardinality of $\mathcal{C}h$ is equivalent to estimating the number of edges in the dual tree $\Gamma$. As observed in Step~2, the graph $\Gamma$ has at most three maximal linear paths. Summing the number of edges along these paths counts each edge at most twice. Combining this with the upper bound on the number of edges in each maximal linear path obtained in Step~3, we find that the number of edges in $\Gamma$ is at most
$$
\frac{3}{2} \cdot \left(2 + \frac{1}{\delta} \right) = \frac{6\delta + 3}{2\delta} \leq \frac{5}{2\delta},
$$
where the last inequality follows from Lemma~\ref{lem:curvaturegap}.
\end{proof}

\begin{rmk}
Since it is a local construction, Lemma~\ref{lem:disk} holds in fact in a much more general settings. We just have to assume that the disk immersion preserves a complex affine structure. Proposition~\ref{prop:combinatorial} and Theorem~\ref{thm:MAIN} will also extend to complex affine surfaces of genus zero that admit a Delaunay triangulation.
\end{rmk}

\subsection{Essentially sharp upper bound on diameter}\label{sub:essential}
We deduce from Lemma~\ref{lem:Delaunay} an upper bound on the \textit{diameter} of a flat sphere.

\begin{proof}[Proof of Corollary~\ref{cor:diameter}]
For any pair of points $M,N$ in the flat sphere, we are going to prove that the distance between $M$ and $N$ induced by the flat metric is at most $(n+1)(\frac{2}{\sqrt{\pi}}+\frac{1}{\sqrt{2\pi\delta}})$.
\par
Unless $M$ and $N$ are already conical singularities, we mark them (the number of singularities is then at most $n+2$) and consider a Delaunay triangulation compatible with this new set of singularities.
\par
The Delaunay edges form a metric graph embedded in $X$. The shortest path between $M$ and $N$ in the graph is formed by at most $n+1$ edges. Lemma~\ref{lem:Delaunay} proves that each Delaunay edge is of length strictly less than $\frac{2}{\sqrt{\pi}}+\frac{1}{\sqrt{2\pi\delta}}$. It follows that the distance between $M$ and $N$ is smaller than $(n+1)(\frac{2}{\sqrt{\pi}}+\frac{1}{\sqrt{2\pi\delta}})$. A bound on the diameter follows.
\end{proof}

\begin{rmk}
Our proof plays the role of an effective version of Theorem 1 of \cite{Chew}. It was used in Section~2.2 of \cite{ACM} to prove that any saddle connection is homotopic to a path in the edges of the Delaunay triangulation of length at most a fixed multiple of the length of the saddle connection.
\end{rmk}

Our upper bound on the diameter is of order $\frac{1}{\sqrt{\delta}}$. In fact, this upper bound is essentially sharp which means that there is a surface whose diameter has a lower bound of order $\frac{1}{\sqrt{\delta}}$. We will construct such a surface in the following example. It also implies that an upper bound on the diameter has to depend on the number of singularities. 

\begin{ex}
Consider a flat sphere $X$ constructed as in Figure~\ref{fig:Sharp}. $l'$ is an arbitrary positive number and $l\in (0, l')$. $X$ has $m$ singularities of curvatures $\epsilon$, $m-1$ singularities of curvatures $-\epsilon$ and two singularities of curvatures $1-\frac{\epsilon}{2}$. Notice that when $\epsilon<\frac{2}{2m+1}$, the curvature gap $\delta$ of $X_{\epsilon}$ is $\frac{\epsilon}{2}$.
    
    \begin{figure}[!htbp]
	\centering
	\begin{minipage}{1\linewidth}
	\includegraphics[width=\linewidth]{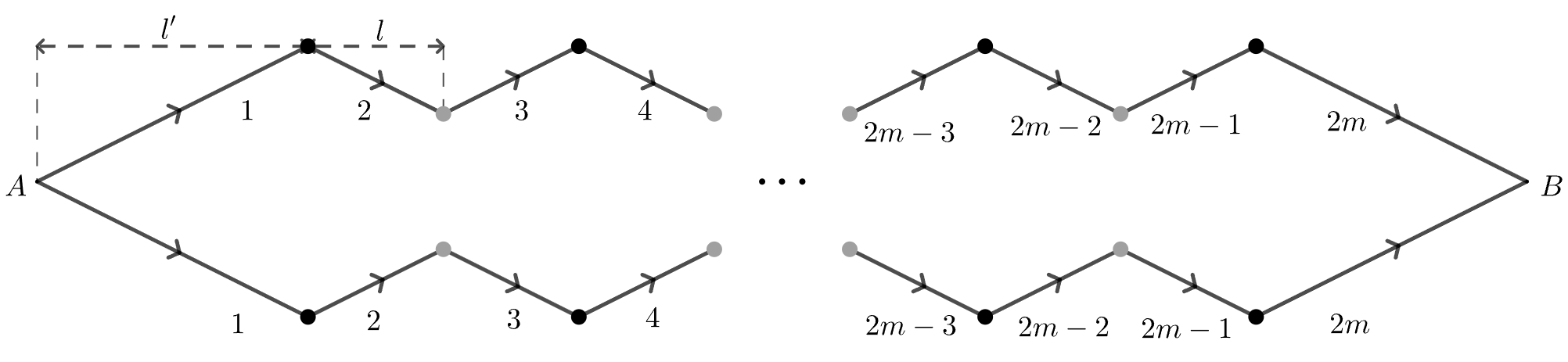}
	\footnotesize
	\end{minipage}
    \caption{Edges of the polygon are glued with respect to numbers and arrows. Curvatures of $A$ and $B$ are both $1-\frac{\epsilon}{2}$, black vertices are singularities of the curvature $\epsilon$ and grey vertices are singularities of the curvature $-\epsilon$.}\label{fig:Sharp}
    \end{figure}
    
Notice that the distance between two singularities of curvatures $1-\frac{\epsilon}{2}$ is realized by the diagonal between them in the polygon. The length of this diagonal is $2l' + 2(m-1)l > 2ml$. The area of the surface is $2[l'^2 + (m-1) (2ll'-l^2)]\tan\frac{\pi\epsilon}{2}<4ml'^2 \epsilon$. 
\par
Hence, the normalized diameter of $X$ is at least $\frac{l}{l'}\sqrt{\frac{m}{\epsilon}}=\frac{l}{l'}\sqrt{\frac{m}{2\delta}}$. In particular, this implies that we cannot hope for a uniform upper bound that does not depend on the number of singularities.    
\end{ex}

\section{Upper bounds}

\subsection{Bounding combinatorial lengths of trajectories}\label{sub:combinatorial}

\begin{defn}
In a flat sphere endowed with a fixed Delaunay triangulation, for any non-periodic trajectory $t$ in $X$, we define the \textit{combinatorial length} of $t$ as the number of segments of $t$ in the complement of $X$ of the union of Delaunay edges.
\end{defn}

In the following proposition, we use the curvature gap to give an upper bound on the combinatorial length of any simple trajectory.

\begin{prop}\label{prop:combinatorial}
We consider a flat sphere $X$ with $n$ conical singularities, a curvature gap $\delta>0$ and a fixed Delaunay triangulation. For any simple trajectory $t$ in $X$, unless $t$ is contained in a Delaunay edge, the combinatorial length of $t$ is at most $\frac{5n}{\delta}$.
\end{prop}

\begin{proof}
We first consider the case where $t$ is a saddle connection. Any Delaunay triangle $\mathcal{T}$ is contained in the image of some locally isometric immersion $f$ of a disk $\mathcal{D}$. Following Lemma~\ref{lem:disk}, the pullback of $t$ by this immersion is formed by at most $\frac{5}{2\delta}$ chords crossing the triangle. 
\par
Since there are $2n-4$ triangles in the Delaunay triangulation, the intersection of $t$ with the complement of Delaunay edges is formed by at most $(2n-4)\frac{5}{2\delta}$ distinct segments. This implies that combinatorial length of $t$ is bounded above by $(2n-4)\frac{5}{2\delta}$.
\par
If $t$ is not a saddle connection, we can mark its endpoints to create conical singularities of angle $2\pi$. This does not change the value of $\delta$ but may increase the number of triangles in the Delaunay triangulation by $2$. The upper bound obtained is $\frac{5n}{\delta}$ instead.
\par
Note that the trajectory cannot be periodic because it would cut out the flat sphere into two components each having a total discrete curvature equal to $1$ (Gauss-Bonnet formula). The curvature gap would therefore be zero.
\end{proof}

We next prove an upper bound on the number of simple saddle connections using Proposition~\ref{prop:combinatorial}.
\par
Kneser introduced in \cite{kne} normal coordinates to represent curves. We define a similar coordinate in our setting. A \textit{simple path} refers to a topological path which connects conical singularities and has no self-intersections or conical singularities in the interior. Given a triangulation $T$, we call a simple path \textit{normal} with respect to $T$ if it intersects with edges of $T$ transversely except at its endpoints and it enters and leaves a triangle of $T$ via different edges. In particular, if a simple saddle connection is not an edge of a Delaunay triangulation, it is normal with respect to the Delaunay triangulation.
\par
Denote by $p$ a normal simple path. For every edge $e$ of $T$, we define a non-negative integer to be the number of intersections of $p$ with the interior of the edge $e$, denoted by $\gamma(e)$. Since there are $3(n-2)$ many edges in $T$, we obtain a vector of $3(n-2)$ non-negative integers in $\mathbb{Z}^{3(n-2)}$ and we call it the \textit{normal coordinate} of $p$.

\begin{lem}\label{lem:normal}
    Given a triangulation $T$, a normal coordinate determines a normal simple path with respect to $T$, which is unique up to isotopy.
\end{lem}
\begin{proof}
    Notice that triangles of $T$ cut the normal path $p$ into sub-paths $p_1,\dots, p_k$. The order of these sub-paths is according to the orientation of $p$. $p_1$ and $p_2$ are special because they start or end in a vertex of $T$. 
    \par
    For a triangle of $T$ that does not contain $p_1$ or $p_k$, we denote edges of the triangle by $e_1,e_2,e_3$. Notice that $p(e_1) + p(e_2) + p(e_3)$ is even and $p(e_i)+p(e_j)\ge p(e_k)$ for $\{i,j,k\} = \{1,2,3\}$. As normal coordinates in \cite{kne}, for fixed $p(e_1),p(e_2),p(e_3)$, there is only one way to pair intersections on different edges by arcs in the triangle such that these arcs do not intersect in the triangle (see Figure~\ref{fig:triangle}).

    \begin{figure}[!htbp]
	\centering
	\begin{minipage}{0.5\linewidth}
	\includegraphics[width=\linewidth]{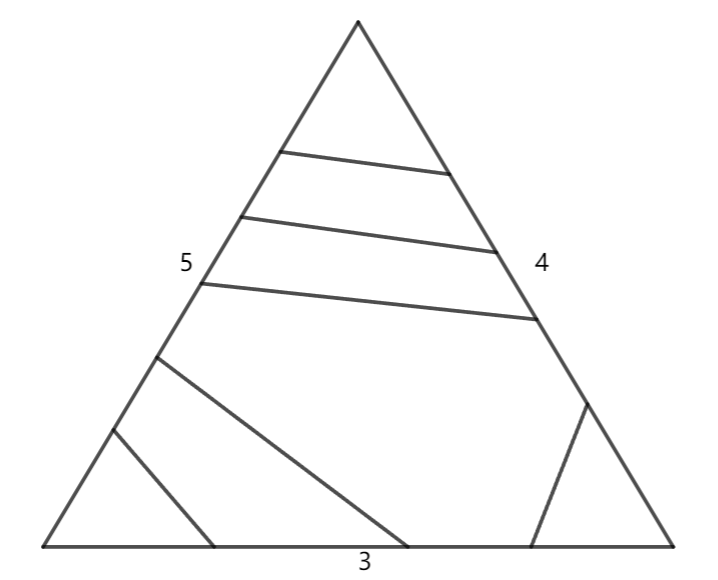}
	\footnotesize
	\end{minipage}
    \caption{The numbers $\{3,4,5\}$ determine the disjoint arcs in the triangle up to isotopy.}\label{fig:triangle}
    \end{figure}
    
    For $p_1$ and $p_k$, there are two possibilities. If they are in two different triangles, for each of the triangles, we can find a way to name edges of triangles by $e_1,e_2,e_3$ such that $p(e_3) = p(e_1) + p(e_2) +1$. If they are in the same triangle, we can name edges of the triangle by $e_1,e_2,e_3$ such that $p(e_3) = p(e_1) + p(e_2) +2$. In either case, once $p(e_1),p(e_2),p(e_3)$ are fixed, there is only one way to pair intersections on different edges by arcs such that these arcs do not intersect in the triangle (see Figure~\ref{fig:specialtriangle}). Hence, a normal coordinate corresponds to an unique way to reconstruct a normal path up to isotopy.

    \begin{figure}[!htbp]
	\centering
	\begin{minipage}{1\linewidth}
	\includegraphics[width=\linewidth]{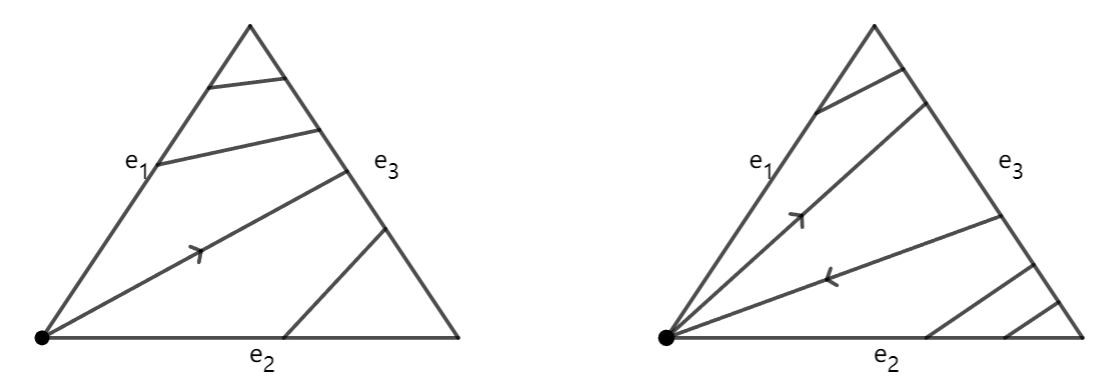}
        \caption{The left figure shows the case when a triangle contains only the first sub-path. The right figure shows the case when a triangle contains the first and the last sub-paths.}\label{fig:specialtriangle}
	\footnotesize
	\end{minipage}
    \end{figure}
\end{proof}

\begin{cor}\label{cor:numbersimplepoly}
    In a flat sphere with $n$ conical singularities and a curvature gap of $\delta>0$, the number of simple saddle connections is at most $\frac{(\frac{5n}{\delta}+3n-7)^{3n-6}}{(3n-7)!} + 3n-6$.
\end{cor}
\begin{proof}
Given an integer $k$, we first prove an upper bound on the number of saddle connections crossing exactly $k$ Delaunay edges. Denote by $\gamma$ such a saddle connection. If $\gamma$ is an edge of the Delaunay triangulation, then it is one of the $3(n-2)$ Delaunay edges. 
\par
If $\gamma$ is not an edge of the Delaunay triangulation, then it is normal with respect to the Delaunay triangulation. It follows that the normal coordinate of $\gamma$ is well-defined. Since $\gamma$ crosses exactly $k$ Delaunay edges, the sum of the integer $\gamma(e)$ over $3(n-2)$ Delaunay edges is equal to $k$. According to Lemma~\ref{lem:normal}, the number of saddle connections that cross exactly $k$ Delaunay edges is equal to the number of ordered $3(n - 2)$-tuples of non-negative integers summing to $k$. This number is given by
$$
\binom{k + 3(n - 2) - 1}{3(n - 2) - 1} < \frac{(k + 3n - 7)^{3n - 7}}{(3n - 7)!}.
$$

Therefore, the total number of saddle connections crossing at most $k$ Delaunay edges is at most 
$$3(n-2) + \sum\limits_{1\le i \le k} \frac{(k+3n-7)^{3n-7}}{(3n-7)!} < 3n-6 + \frac{(k+3n-7)^{3n-6}}{(3n-7)!}.$$ 
Since Proposition~\ref{prop:combinatorial} tells us that $k< \frac{5n}{\delta}$, we obtain the upper bound $\frac{(\frac{5n}{\delta}+3n-7)^{3n-6}}{(3n-7)!} + 3n-6$. 
\end{proof}

\subsection{Trajectories with self-intersections}\label{sub:Self}

In order to bound the combinatorial length of non-simple trajectories, we combine geometric estimates with purely topological bounds on isotopically disjoint loops. We start by giving a rigorous definition of the number of self-intersections of a trajectory (see \cite{Ba} for the notion of self-intersections for geodesics on hyperbolic surfaces). 

\begin{defn}\label{defn:transverse}
At any self-intersection point $p$ of a trajectory $t$ in a flat sphere $X$, we say that two tangent vectors to $t$ (for the orientation of $t$) form a \textit{transverse pair} if they span the tangent plane at $p$. The total number $\iota(t, t)$ of transverse pairs (summed on the set of self-intersection points of $t$) is the \textit{self-intersection number} of $t$.
\end{defn}

The following topological lemma follows directly from the pairs-of-pants decomposition.

\begin{lem}\label{lem:isoloops}
In a topological sphere with $n \geq 3$ punctures, the maximal number of  isotopically disjoint homotopically nontrivial loops is $2n-3$.
\end{lem}

A monogonal trajectory $\gamma$ (see Section~\ref{sub:monogon}) decomposes a flat sphere $X$ into two connected components:
\begin{itemize}
    \item the \textit{inner component} $X_{\gamma}^{-}$ where the total curvature is strictly less than $1$;
    \item the \textit{outer component} $X_{\gamma}^{+}$ where the total curvature is strictly greater than $1$.
\end{itemize}

The following lemma give bounds for a trajectory entirely contained in the inner component relative to a monogonal trajectory. Notice that the result holds as long as the bound on combinatorial lengths of simple trajectories on $X$ (possibly zero curvature gap) is finite.

\begin{lem}\label{lem:innertraj}
We consider a flat sphere $X$ with a monogonal trajectory $\gamma$ of interior angle $\alpha$ and length $L$. Assuming that the bound $l$ on combinatorial lengths of simple trajectories is finite, any trajectory $t$ contained in $X_{\gamma}^{-}$ satisfies one of the following statements:
\begin{enumerate}
    \item $t$ does not contain any monogonal trajectory that would be isotopic to $\gamma$;
    \item the combinatorial length of $t$ is at most $2l$. Besides, its metric length is at most $\frac{L}{\cos(\alpha/2)}$.
\end{enumerate}
\end{lem}

\begin{proof}
We assume that $t$ contains a monogonal trajectory $\gamma'$ that is isotopic to $\gamma$. Let $\mathcal{C}$ be the topological cylinder cut out by $\gamma \cup \gamma'$. The universal cover $\Tilde{\mathcal{C}}$ is bounded by an inner arc (by developing $\gamma'$) and an outer arc (by developing $\gamma$). Since every interior angle of the inner arc is strictly bigger than $\pi$, it follows that the maximal extension in both directions of monogonal trajectory $\gamma'$ leaves $\mathcal{C}$ by crossing $\gamma$ (see Figure~\ref{fig:topocylinder}).

    \begin{figure}[!htbp]
	\centering
	\begin{minipage}{0.5\linewidth}
	\includegraphics[width=\linewidth]{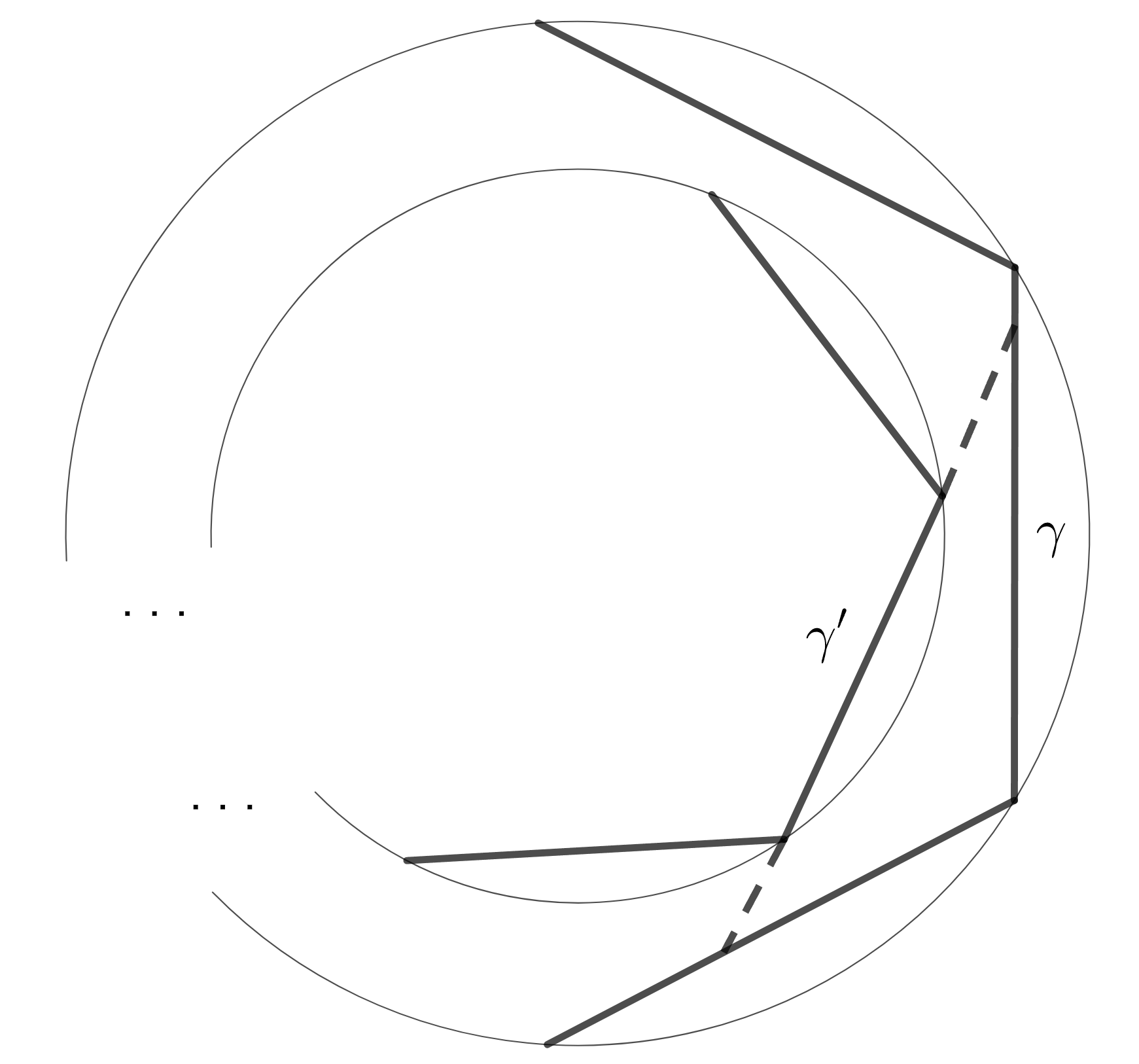}
	\footnotesize
	\end{minipage}
    \caption{The maximal extension of $\gamma'$ crosses $\gamma$.}\label{fig:topocylinder}
    \end{figure}

We draw $\Tilde{\mathcal{C}}$ inside the universal cover of $\mathbb{C}^{\ast}$ and identify the origin with the center of the rotation of the holonomy of the flat metric along loop $\gamma$ (and also $\gamma'$). It follows that the maximal extension of $\gamma'$ is contained in a chord of a disk of radius $\frac{L}{2\cos(\alpha/2)}$ (see Section~\ref{sub:monogon}). Consequently, the metric length of $t$ is at most $\frac{L}{\cos(\alpha/2)}$.
\par
The behaviour of trajectories in a flat annulus (the case of $\mathcal{C}$ is similar) is described in Section~\ref{sub:trajannulus}. The maximal extension of $\gamma'$ in each direction is simple, see Figure~\ref{fig:crossing}. It follows that $t$ is formed by at most two simple trajectories. Since the combinatorial length of a simple trajectory is at most $l$ by the condition, the desired bound follows immediately.
\end{proof}

Now we are able to generalize the bound of Proposition~\ref{prop:combinatorial} to self-intersecting trajectories.

\begin{prop}\label{prop:SIcombinatorial}
We consider a flat sphere $X$ with $n$ conical singularities and a fixed Delaunay triangulation. Assuming that the bound $l$ on combinatorial lengths of simple trajectories in $X$ is finite. For a trajectory $t$ in $X$, the combinatorial length $s$ of $t$ satisfies $s \leq 4l(n-1)\sqrt{|\iota(t, t)|}+4l$  
where $\iota(t, t)$ is the self-intersection number of trajectory $t$.
\end{prop}

\begin{proof}
If the combinatorial length $s$ of $\gamma$ is less than $4l$,  then the conclusion $s \leq 4l(n-1)\sqrt{|\iota(t, t)}+4l$ holds automatically. Thus, we assume that $s\ge 4l$ in the following.

We decompose the trajectory $t$ into finitely many segments $t_1, \ldots, t_m, t_{m+1}$, where:
\begin{itemize}
    \item Each $t_i$ for $1 \le i \le m$ has combinatorial length exactly $2l$;
    \item The final segment $t_{m+1}$ has combinatorial length strictly less than $2l$ (possibly zero).
\end{itemize}
This decomposition ensures that the number $m$ satisfies
\begin{equation}\label{equ:loweronm}
    m = \left\lfloor \frac{s}{2l} \right\rfloor \ge \frac{s}{2l} - 1.
\end{equation}

We are going to give an upper bound on $m$ in terms of $\iota(t, t)$.

If $m < 2n - 2$, then $s \le 2l(m+1) \le 2l(2n - 3 + 1) = 4l(n - 1).$ By the assumption that $s \ge 4l$ and that $l$ is an upper bound for simple trajectories, we have $\iota(t,t) \ge 1$. Therefore, the inequality $s \le 4l(n - 1)\sqrt{|\iota(t,t)|} + 4l$ holds automatically. Hence, in the following we restrict to the case $m \ge 2n - 2$.

Since the combinatorial length of $t_i$ is $2l>l$, each $t_i$ for $1 \le i \le m$ is not simple. In particular, $t_{i}$ contains a monogonal trajectory $\gamma_{i}$ (see Section~\ref{sub:monogon}). Besides, Lemma~\ref{lem:innertraj} proves that for any pair $i,j$ with $i \neq j$, one of two statements holds:
\begin{itemize}
    \item $t_{i}$ intersects $t_{j}$;
    \item $\gamma_{i}$ and $\gamma_{j}$ are isotopically distinct.
\end{itemize}

For each subset $S \subset \{1, \dots, m\}$ of cardinality $|S| = 2n - 2$, we claim that there exists at least one pair $i \neq j$ in $S$ such that the segments $t_i$ and $t_j$ intersect. Indeed, if all such pairs were disjoint, the corresponding monogonal trajectories $(\gamma_i)_{i \in S}$ would form a family of $2n - 2$ isotopically disjoint loops, contradicting Lemma~\ref{lem:isoloops}.

There are $\binom{m}{2n - 2}$ such subsets $S$. For a fixed pair $i \ne j$, the number of subsets $S$ that contain both $i$ and $j$ is $\binom{m - 2}{2n - 4}$. Therefore, the total number of intersecting pairs among the segments $t_1, \ldots, t_m$ is at least
$$
\frac{\binom{m}{2n - 2}}{\binom{m - 2}{2n - 4}} = \frac{m(m - 1)}{(2n - 2)(2n - 3)} \ge \left( \frac{m - 1}{2n - 2} \right)^2.
$$
This implies that
$$
\iota(t, t) \ge \left( \frac{m - 1}{2n - 2} \right)^2.
$$
Rearranging, we obtain
$$
m \leq 2(n - 1)\sqrt{\iota(t, t)} + 1.
$$
Since inequality~\eqref{equ:loweronm} gives $m \ge \frac{s}{2l} - 1$, we deduce that
$$
s \le 2l m + 2l \le 4l(n - 1)\sqrt{\iota(t, t)} + 4l.
$$
\end{proof}

\begin{proof}[Proof of Theorem~\ref{thm:MAIN}]
We are going to prove an upper bound on the number of saddle connections crossing exactly $s$ Delaunay edges. Since there is at most one saddle connection in each homotopy class of topological arcs, an oriented saddle connection $\gamma$ is completely characterized by:
\begin{itemize}
    \item the triangle corner where $\gamma$ starts;
    \item for each $1 \leq i \leq k-1$, whether the trajectory leaves the triangle through the left or the right side after the $i$-th crossing with a Delaunay edge.
\end{itemize}
Forgetting the orientation of the saddle connection, we deduce that there are at most $(3n-6)2^{s-1}$ saddle connections crossing exactly $s$ Delaunay edges.
\par
Summing these upper bounds (including the number of saddle connections coinciding with a Delaunay edge), we obtain that the total number of saddle connections crossing at most $s$ saddle connections is at most $(3n-6)2^{s}$.
\par
Proposition~\ref{prop:SIcombinatorial} provides the upper bound on the number $s$ of crossings with Delaunay edges for a saddle connection with at most $k$ self-intersections, which is $s\le 4l(n-1)\sqrt{\iota(t, t)}+4l$.
Proposition~\ref{prop:combinatorial} implies that the bound $l$ on combinatorial lengths of simple trajectories is at most $\frac{5n}{\delta}$. Hence, $s \leq \frac{20n(n-1) \sqrt{k} + 20n}{\delta}$. When $k =0$, the bound is given by Corollary~\ref{cor:numbersimplepoly}.
\end{proof}

\subsection{Bounding metric lengths of trajectories}\label{sub:length}

We have an upper bound on combinatorial lengths of trajectories with bounded number of self-intersections (see Proposition~\ref{prop:SIcombinatorial}). Combined with a control on the maximal length of Delaunay edges (see Lemma~\ref{lem:Delaunay}), we get an explicit upper bound on the maximal length of a trajectory with bounded number of self-intersections (including simple trajectories). Remarkably, this bound does not depend on the systole (length of the smallest simple saddle connection) but only on the curvature gap. The bound is uniform on the moduli space of flat sphere with prescribed conical angles.

\begin{proof}[Proof of Theorem~\ref{thm:MAIN2}]
Following Lemma~\ref{lem:Delaunay}, the length $L$ of any Delaunay edge in such a flat sphere $X$ satisfies $L^{2} < \frac{4}{\pi}+\frac{1}{2\pi\delta}$ and thus $L<\frac{2}{\sqrt{\pi}}+\frac{1}{\sqrt{2\pi\delta}}$.
\par
For a simple trajectory $t$ of $X$ that is not contained in a Delaunay edge, we deduce from the proof of Proposition~\ref{prop:combinatorial} that $t$ is formed by at most $\frac{5n}{\delta}$ segments, each being contained in a unique Delaunay triangle. We obtain immediately that the metric length of $t$ is less than $\frac{10n}{\delta\sqrt{\pi}}+\frac{5n}{\delta^{3/2}\sqrt{2\pi}}$.
\par
For a self-intersecting trajectory $t$ with at most $k$ self-intersections, Proposition~\ref{prop:SIcombinatorial} implies that the combinatorial length is at most $ 4l(n-1)\sqrt{\iota(t, t)}+4l$ and Proposition~\ref{prop:combinatorial} implies that  $l\le \frac{5n}{\delta}$. Combining with $L<\frac{2}{\sqrt{\pi}}+\frac{1}{\sqrt{2\pi\delta}}$, the metric length is less than 
$
\frac{40n(n-1)\sqrt{k}+40n}{\delta\sqrt{\pi}}
+
\frac{20n(n-1)\sqrt{k}+20n}{\delta^{3/2}\sqrt{2\pi}}.
$
\end{proof}

\subsection{Families of homotopic regular closed geodesics}\label{sub:parallel}

A \emph{regular closed geodesic} on a flat sphere $X$ is a geodesic that does not pass through any singularity and returns to its starting point with the same tangent direction. 

In the following, homotopies on the sphere are always considered relative to the singularities. Hence, if two regular closed geodesics are homotopic, then they bound an immersed flat cylinder containing no singularities. A \emph{family of homotopic regular closed geodesics} is an immersed open flat cylinder foliated by regular closed geodesics.

A family of homotopic regular closed geodesics is said to be \emph{maximal} if it contains all regular closed geodesics that are homotopic to them.

\begin{lem}\label{lem:parallelintersection}
Two homotopic regular closed geodesics on a flat sphere have the same self-intersection number.
\end{lem}

\begin{proof}
Let $t$ and $t'$ be two homotopic regular closed geodesics on a flat sphere $X$. By definition, they bound an immersed flat cylinder $I \times S^1$ for some interval $I$. We orient both $t$ and $t'$ consistently with the standard orientation of $S^1$.

With this orientation, the left and right sides of $t$ are well-defined, and $t'$ lies entirely on one side of $t$. We assume, without loss of generality, that it lies on the left. For each transverse pair at a self-intersection of $t$, the local configuration is as illustrated in the left part of Figure~\ref{fig:evenintersection}.

This local picture defines a map that sends each transverse pair of $t$ to a transverse pair of $t'$. Reversing the roles of $t$ and $t'$ yields the inverse map by the same analysis. Hence, the map is a bijection between the transverse pairs of $t$ and those of $t'$, and in particular, $\iota(t, t) = \iota(t', t')$.
\end{proof}

By the above lemma, we define the \emph{self-intersection number} of a family of homotopic regular closed geodesics as the self-intersection number of a regular closed geodesic in the family.

The upper bounds in Theorem~\ref{thm:MAIN} also apply to the count of regular closed geodesics. For clarity, given an integer $k$, we denote by $N^{sc}(X,k)$ the number of saddle connections on $X$ with self-intersection number at most $k$, and by $N^{cg}(X,k)$ the number of maximal families of homotopic regular closed geodesics on $X$ with self-intersection number at most $k$.

\begin{cor}\label{cor:countclosedgeodesics}
In a flat sphere $X$, we have $N^{cg}(X,k) \leq N^{sc}(X,k)$ for any $k\in\mathbb{N}$. In particular, if $X$ has $n$ singularities and a positive curvature gap $\delta$, then $N^{cg}(X,k) \leq (3n - 6) \cdot 2^{s}$, where $s = \frac{20n(n - 1)\sqrt{k} + 20n}{\delta}$.
\end{cor}

\begin{proof}
To prove the first claim, we fix once and for all an orientation for each saddle connection on $X$. We define $\widehat{N}^{sc}(X,k)$ to be the number of pairs $(t,s)$, where $t$ is a saddle connection on $X$ with self-intersection number at most $k$, and $s\in\{L,R\}$ specifies the choice of the left or right side of $t$ relative to the chosen orientation.

Since each saddle connection has exactly two sides, we have $\widehat{N}^{sc}(X,k)=2N^{sc}(X,k)$. Thus, to prove $N^{cg}(X,k)\le N^{sc}(X,k)$, it suffices to show that $2N^{cg}(X,k)\le \widehat{N}^{sc}(X,k).$

Let $\mathcal{C}_k$ be the set of triples $(C,t,s)$, where
\begin{itemize}
    \item $C$ is a maximal family of homotopic regular closed geodesics on $X$ whose members have self-intersection number at most $k$ (equivalently, an immersed maximal flat cylinder with this property);
    \item $t$ is a saddle connection contained in the boundary of the cylinder determined by $C$;
    \item $s\in\{L,R\}$ is a side of $t$ such that the cylinder determined by $C$ is adjacent to $t$ from the side $s$.
\end{itemize}
For each maximal family $C$ with self-intersection number $\le k$, the flat cylinder determined by $C$ has exactly two boundary components. For each boundary component (which may contain several saddle connections), one can choose a saddle connection $t$ contained in it and let $s\in\{L,R\}$ be the side of $t$ from which $C$ is adjacent. This defines a triple $(C,t,s)\in \mathcal{C}_k$. Therefore, $2N^{cg}(X,k)\le |\mathcal{C}_k|$.

There is a natural map $\mathcal{C}_k\to \{(t,s)\}$ given by $(C,t,s)\mapsto (t,s)$. Moreover, any saddle connection $t$ appearing in the boundary of the cylinder of $C$ has self-intersection number bounded above $k$. Therefore, the image consists of pairs counted by $\widehat{N}^{sc}(X,k)$.

Finally, for a fixed pair $(t,s)$, there is at most one maximal family $C$ whose cylinder is adjacent to $t$ from the side $s$. It follows that the above map is injective, and hence $|\mathcal{C}_k|\le \widehat{N}^{sc}(X,k)$.
Combining the two inequalities gives
\[
2N^{cg}(X,k)\le |\mathcal{C}_k|\le \widehat{N}^{sc}(X,k).
\]
Since $\widehat{N}^{sc}(X,k)=2N^{sc}(X,k)$, we have $N^{cg}(X,k)\le N^{sc}(X,k)$.

For the second claim, Theorem~\ref{thm:MAIN} provides an upper bound for $N^{sc}(X,k)$. Together with $N^{cg}(X,k)\le N^{sc}(X,k)$, this yields the stated upper bound for $N^{cg}(X,k)$.
\end{proof}

\section{Application to irrational billiard}\label{sec:billiard}

Let $P$ be a polygon in the plane. We recall a classical construction from~\cite{FoxKer, KZ} that associates to $P$ a flat sphere. Take two copies of $P$ and identify their boundaries. The resulting flat sphere is denoted by $X_P$ and is called the \emph{pillowcase of the polygon $P$}. This construction induces a quotient map $\pi: X_P \to P$, where $\pi$ maps each copy to $P$.

Let $\Lambda$ be the set of vertices of $P$. Note that $v\in \Lambda$ gives rise to a conical singularity of $X_v$. We define the \textit{discrete curvature} of $P$ at $v$, denote by $k_v$, to be the curvature of the singularity corresponding to $v$ in $X_P$. More precisely, if $\theta$ is the interior angle of $P$ at $v$, then the cone angle of the singularity of $v$ in $X_P$ is $2\theta$. Then the curvature of the singularity is
$$
k_v =\frac{2\pi - 2\theta}{2\pi} = \frac{\pi - \theta}{\pi}.
$$

We then define the \textit{curvature gap} $\delta$ of $P$ to be the curvature gap of $X_P$. Equivalently, $\delta = \inf\limits_{I \subset \Lambda} \left| 1 - \sum\limits_{v \in I} k_v \right|$.

\subsection{Unfolding billiard paths}
We explain how to unfold a billiard path on $P$ on the pillowcase $X_P$ in the following. 
\par
We first consider the case of a generalized diagonal $\gamma$, and denote by $(s_1, \ldots, s_n)$ the sequence of geodesic segments ordered according to the parameterization of $\gamma$. We lift $s_1$ to a chosen copy of $P$ in $X_P$, then lift $s_2$ to the other copy, and continue in this alternating manner. The resulting path is a saddle connection on $X_P$, denoted by $t_\gamma$.
\par
Second, we consider the case of a periodic billiard path $\gamma$. Denote by $(s_1, \ldots, s_n)$ the sequence of geodesic segments of $\gamma$, ordered as in the previous case. We define the \emph{period} of $\gamma$ to be the minimal number of reflections after which the path returns to its starting point with the same tangent direction, which in this case is simply $n$.
\begin{enumerate}
    \item If the period of $\gamma$ is an even number $2k$, then we unfold $\gamma$ as in the case of generalized diagonal. The final segment $s_{2k}$ is unfolded in a different copy of $P$ on $X_P$ from that of $s_1$. Hence, $\gamma$ is unfolded into a regular closed geodesic, denoted by $t_\gamma$.
    \item If $\gamma$ has odd period $2k+1$, then the segment $s_{2k+1}$ is lifted to the same copy of $P$ as $s_1$. To obtain a closed geodesic on $X_P$, we unfold the sequence $(s_1, \ldots, s_{2k+1})$ once more. The resulting regular closed geodesic is again denoted by $t_\gamma$.
\end{enumerate}
We refer to $t_\gamma$ as the \emph{unfolding of $\gamma$} on the pillowcase $X_P$.

\begin{rmk}\label{rmk:relation}
Note that in the case of generalized diagonals or periodic billiard paths of even period, the unfolded path $t_\gamma$ depends on the choice of the initial copy of $P$ in which the segment $s_1$ lies. As a result, there are two distinct unfoldings of $\gamma$ on $X_P$.
\par
Moreover, the Euclidean length of $t_\gamma$ is equal to that of $\gamma$ if $\gamma$ is a generalized diagonal or a periodic billiard path of even period; otherwise, it is twice the length of $\gamma$.
\end{rmk}

Two periodic billiard paths $\gamma$ and $\gamma'$ in $P$ are said to be \emph{parallel} if there exist unfoldings $t_\gamma$ and $t_{\gamma'}$ on $X_P$ such that $t_\gamma$ and $t_{\gamma'}$ are homotopic regular closed geodesics.
\par
A \emph{family of parallel periodic billiard paths} is a collection of periodic billiard paths in $P$ whose unfoldings form a family of homotopic regular closed geodesics on $X_P$. Such a family is said to be \emph{maximal} if its unfolding is a maximal family of homotopic closed geodesics on $X_P$.

\subsection{Self-intersection numbers of billiard paths}\label{sec:intersectionnumberbilliardpaths}

We aim to count the number of generalized diagonal and periodic billiard paths of bounded ``self-intersection number''. For that purpose, we define the \emph{self-intersection number} of a generalized diagonal or a periodic billiard path as follows.

Let $\gamma$ be a billiard path on $P$. We first define \emph{transverse pairs} at self-intersection points:
\begin{itemize}
    \item If a self-intersection point $p$ lies in the interior of $P$, transverse pairs at $p$ are defined as in Definition~\ref{defn:transverse}.
    \item If a self-intersection point $p$ lies on $\partial P$, we consider only the tangent vectors of $\gamma$ at $p$ that point into the interior of $P$. A \emph{transverse pair} at $p$ is then defined as a pair of such tangent vectors that are linearly independent.
\end{itemize}
We then define the \emph{self-intersection number} $\iota(\gamma, \gamma)$ of a billiard path $\gamma$ as the total number of transverse pairs at all its self-intersection points.

However, the case of periodic billiard paths requires additional care. We want to count the maximal family of parallel periodic billiard paths. However, these numbers may vary along the family, see Figure~\ref{fig:odd}.

    \begin{figure}[!htbp]
	\centering
	\begin{minipage}{0.5\linewidth}
	\includegraphics[width=\linewidth]{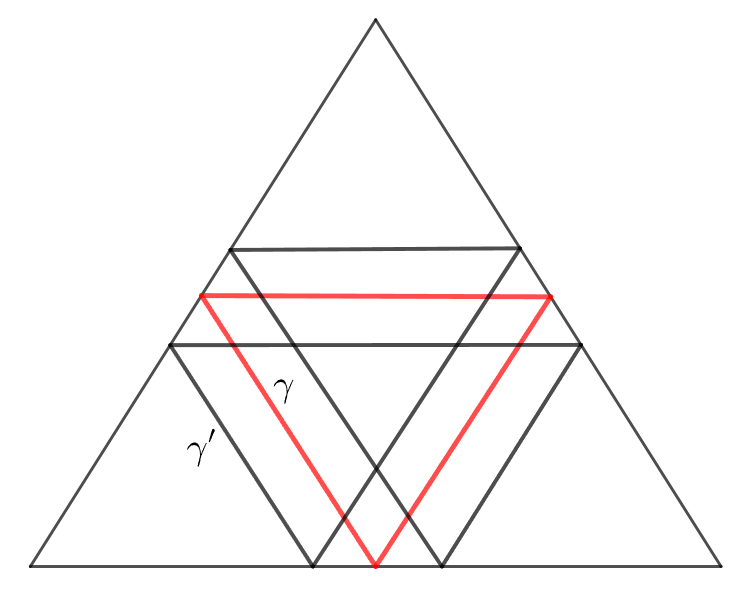}
	\footnotesize
	\end{minipage}
    \caption{The periodic billiard path $\gamma$ is simple, while the parallel path $\gamma'$ has self-intersection number 3.}\label{fig:odd}
    \end{figure}

\begin{lem}\label{lem:selfintersectionnumber}
Let $P$ be a polygonal billiard, and let $\mathcal{F}$ be a maximal family of parallel periodic billiard paths in $P$. Then:
\begin{enumerate}
    \item The family $\mathcal{F}$ contains at most one periodic billiard path of odd period. If $\mathcal{F}$ contains such a path, then the period of any other path in the family is equal to twice the period of the odd one; otherwise, all paths in $\mathcal{F}$ have the same period.
    
    \item Any two parallel periodic billiard paths of even period in $\mathcal{F}$ have the same self-intersection number.
    
    \item If $\mathcal{F}$ contains a periodic billiard path $\gamma$ of odd period $T$, then for any other path $\gamma' \in \mathcal{F}$, we have
    $$
    \iota(\gamma', \gamma') = 4\iota(\gamma, \gamma) + T.
    $$
\end{enumerate}
\end{lem}

\begin{rmk}
We emphasize that the above result crucially depends on our definition of self-intersection number as the total number of transverse pairs. The formula does not hold if one defines the self-intersection number merely as the number of self-intersection points.
\end{rmk}

\begin{proof}
Claim~(1) follows from Theorem~1 in~\cite{Galperin1995}.
\par
For~(2), the proof is similar to that of Lemma~\ref{lem:parallelintersection}. Let $\gamma$ and $\gamma'$ be two parallel periodic billiard paths of even period. Let $(s_1, \ldots, s_T)$ denote the sequence of geodesic segments of $\gamma$ in $P$, ordered according to the parametrization of $\gamma$, and let $(s'_1, \ldots, s'_T)$ be the corresponding sequence for $\gamma'$.
\par
Note that each segment $s_i$ of $\gamma$ is parallel to segment $s'_i$ in $\gamma'$. Therefore, the local configuration of a transverse pair at a self-intersection point of $\gamma$ in the interior of $P$ is illustrated in the left part of Figure~\ref{fig:evenintersection}, while the case of a boundary self-intersection point is shown in the right part. These configurations imply the existence of a map sending each transverse pair of $\gamma$ to a transverse pair of $\gamma'$.
\par
Suppose this map is not injective. Then there exist two distinct transverse pairs of $\gamma$ that map to the same transverse pair of $\gamma'$. This would imply that there are two segments of $\gamma'$  are identical, which is impossible. Hence, the map is injective. 
\par
The same analysis applies when we reverse the roles of $\gamma$ and $\gamma'$. Hence, the above map is indeed a bijection from the transverse pairs of $\gamma$ to those of $\gamma'$. In particular, $\gamma$ and $\gamma'$ have the same self-intersection number.

    \begin{figure}[!htbp]
	\centering
	\begin{minipage}{\linewidth}
	\includegraphics[width=\linewidth]{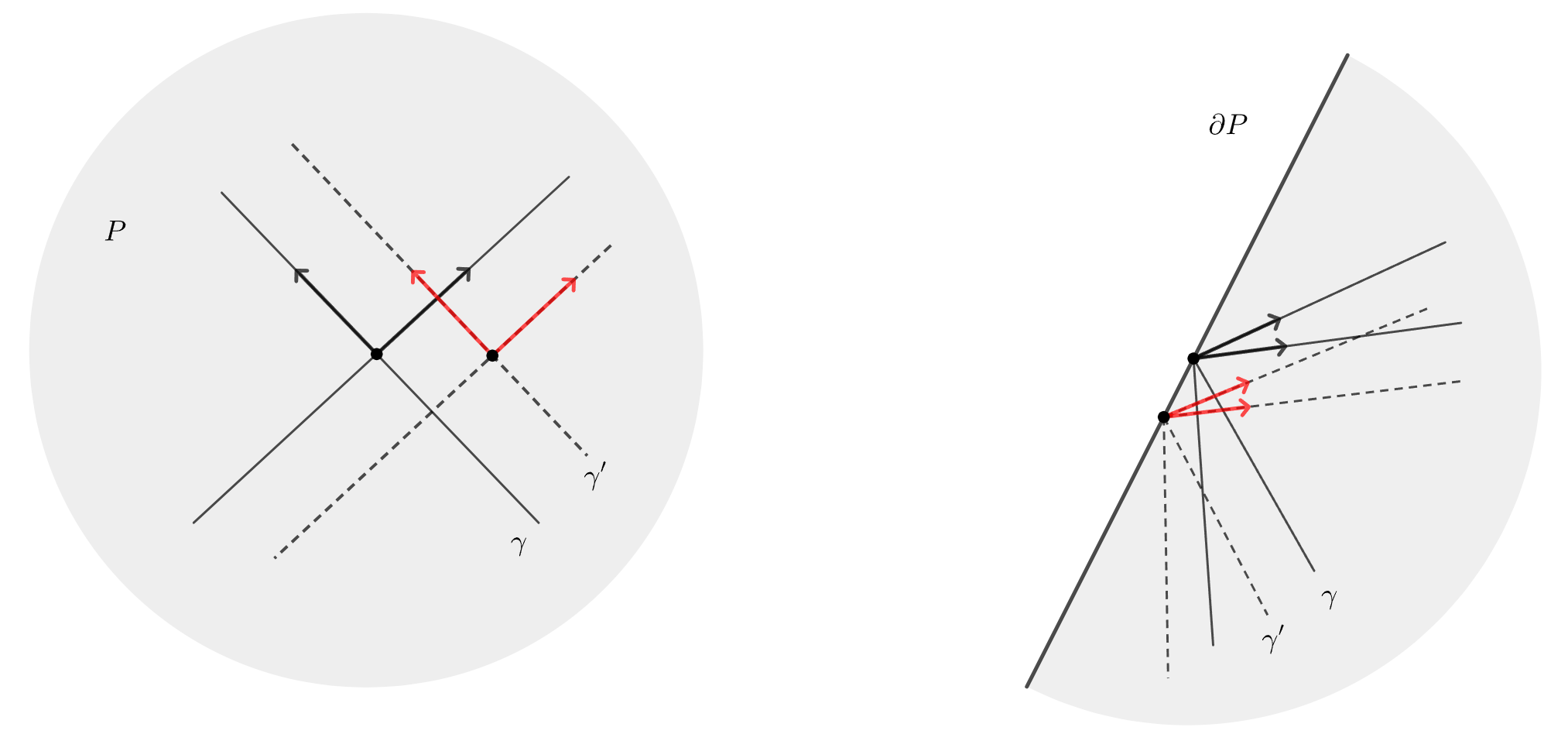}
	\footnotesize
	\end{minipage}
    \caption{A transverse pair of $\gamma$ (solid segments) corresponds to the red transverse pair of a parallel path $\gamma'$ (dotted segments). }\label{fig:evenintersection}
    \end{figure}
    
For~(3), let $\gamma$ be a periodic billiard path of odd period $T$, and let $\gamma'$ be a periodic billiard path parallel to $\gamma$. According to~(1), the period of $\gamma'$ is $2T$. Let $(s_1, \ldots, s_T)$ denote the sequence of geodesic segments of $\gamma$ in $P$, and let $(s'_1, \ldots, s'_{2T})$ denote the corresponding sequence of segments of $\gamma'$.
\par
Note that every time $\gamma$ reflects at a boundary edge, the path $\gamma'$ switches to the opposite side of $\gamma$. Since the period $T$ of $\gamma$ is odd, when $\gamma$ returns to the starting point of $s_1$, the path $\gamma'$ arrives at a segment $s'_{T+1}$ that is parallel to $s_1$ but lies on the opposite side of $s'_1$. This implies that each segment $s_i$ of $\gamma$ corresponds to two parallel segments $s'_i$ and $s'_{T+i}$ in $\gamma'$, while each segment $s'_j$ of $\gamma'$ is parallel to a unique segment $s_j$ (or $s_{j-T}$, if $j > T$) of $\gamma$.

    \begin{figure}[!htbp]
	\centering
	\begin{minipage}{\linewidth}
	\includegraphics[width=\linewidth]{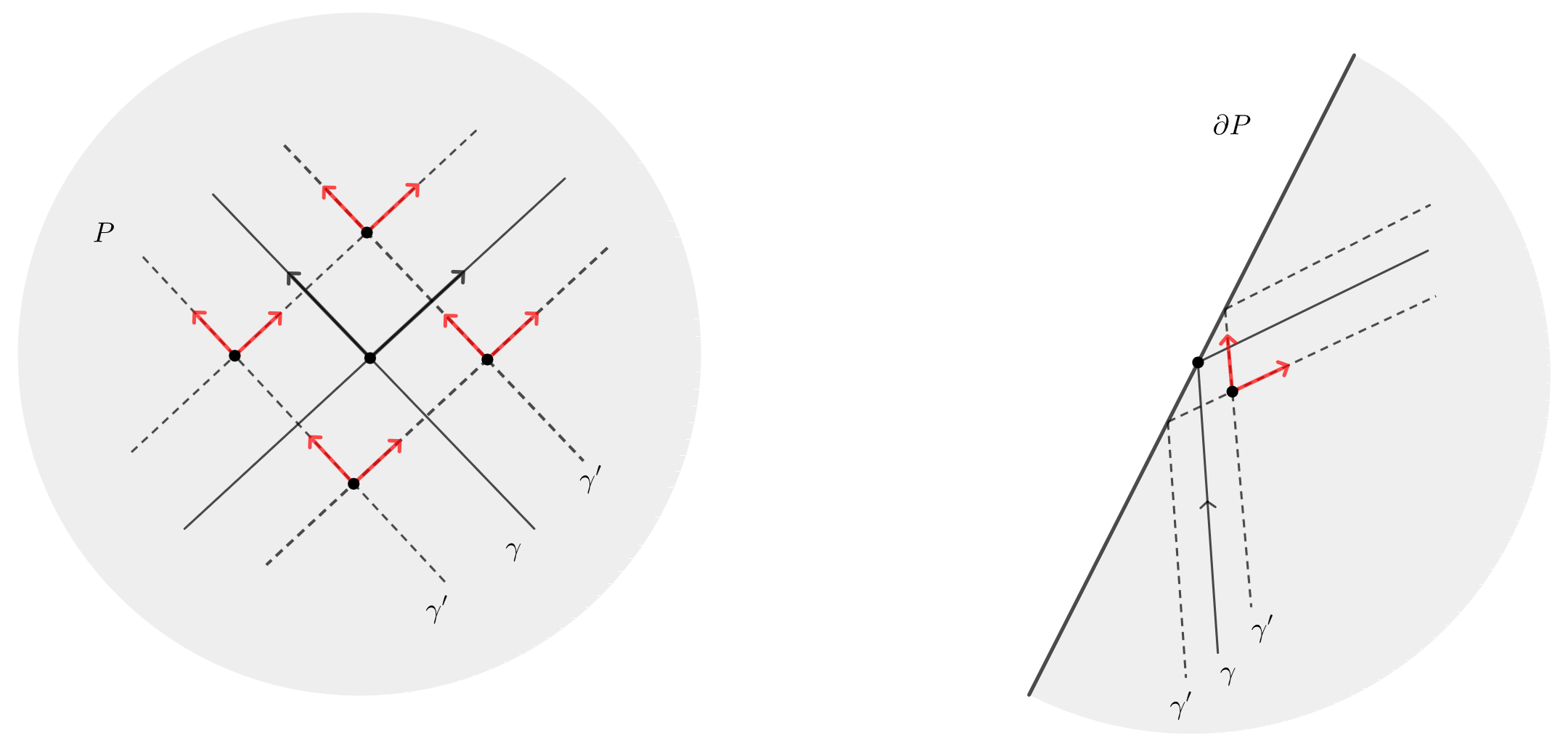}
	\footnotesize
	\end{minipage}
    \caption{Left: a transverse pair of $\gamma$ (solid segments) contributes four red transverse pairs to a parallel path $\gamma'$ (dotted segments). Right: a reflection point of $\gamma$ gives rise to a transverse pair of $\gamma'$.}\label{fig:oddintersection}
    \end{figure}

Therefore, for each transverse pair of $\gamma$ in the interior of $P$, the local configuration is as illustrated in the left part of Figure~\ref{fig:oddintersection}. Each such pair contributes four distinct transverse pairs to $\gamma'$. 
\par
Next, consider a reflection of $\gamma$ at the boundary. As shown in the right part of Figure~\ref{fig:oddintersection}, each such reflection also contributes a transverse pair of $\gamma'$. Moreover, a transverse pair of $\gamma$ at a boundary point of $P$ contributes six transverse pairs to $\gamma'$, two of which arise from the reflections of $\gamma$ at this point, see Figure~\ref{fig:oddintersectionatboundary}.

    \begin{figure}[!htbp]
	\centering
	\begin{minipage}{0.8\linewidth}
	\includegraphics[width=\linewidth]{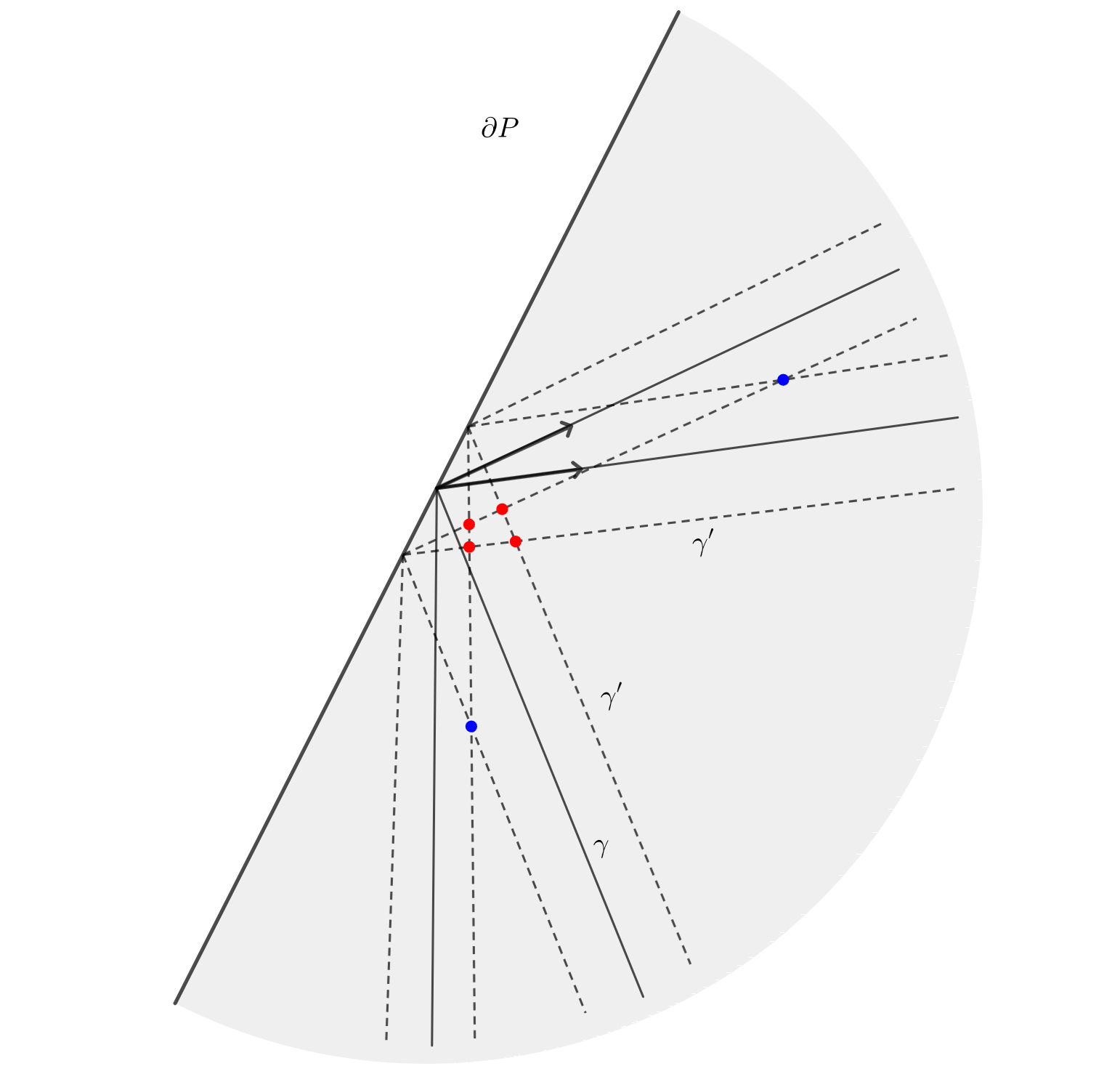}
	\footnotesize
	\end{minipage}
    \caption{The solid segments show a transverse pair of $\gamma$ at the boundary of the polygon $P$, and the dotted segments represent the parallel path $\gamma'$. At each colored point, there is a transverse pair of $\gamma'$. The four transverse pairs at the red points arise from the transverse pair of $\gamma$ at $\partial P$, while the two transverse pairs at the blue points arise from the reflections of $\gamma$.}\label{fig:oddintersectionatboundary}
    \end{figure}

Furthermore, suppose that two distinct local configurations contribute to the same transverse pair of $\gamma'$. This would imply that two geodesic segments of $\gamma'$ are identical, which is impossible.
\par
Conversely, since each segment $s'_j$ of $\gamma'$ is parallel to one segment $s_j$ (or $s_{j-T}$) of $\gamma$, a similar analysis of the local configurations implies that each transverse pair of $\gamma'$ arises from one of the cases described above. 
\par
Combining the above observations, we obtain the formula $\iota(\gamma', \gamma') = 4\iota(\gamma, \gamma) + T$,
where the number of reflections of $\gamma'$ is equal to the period $T$.
\end{proof}

By the lemma above, the following definition of self-intersection number is well-defined.
\begin{defn}\label{def:intersectionparallelperiod}
 The \emph{self-intersection number} of a maximal family of parallel periodic billiard paths is defined to be the self-intersection number of any periodic billiard path of even period in the family.
\end{defn}

\subsection{Proof of Theorem~\ref{thm:MAIN3}}
Let $P$ be a polygonal billiard. For a given $k \in \mathbb{N}$, we define $N^{diag}(P,k)$ to be the number of generalized diagonals in $P$ with self-intersection number at most $k$, and $N^{per}(P,k)$ to be the number of maximal families of periodic billiard paths in $P$ with self-intersection number at most $k$.

Recall from Section~\ref{sub:parallel} that $N^{sc}(X,k)$ denotes the number of saddle connections on a flat sphere $X$ with self-intersection number at most $k$, and $N^{cg}(X,k)$ denotes the number of maximal families of homotopic regular closed geodesics on $X$ with self-intersection number at most $k$.

\begin{proof}[Proof of Theorem~\ref{thm:MAIN3}]
Denote by $P$ a polygon with $n$ vertices and a curvature gap $\delta > 0$. Let $X_P$ be the pillowcase of $P$. By our conditions on $P$, we know that $X_P$ also has a positive curvature gap $\delta$.
\par
We claim that for any generalized diagonal or periodic billiard path $\gamma$ of even period in $P$, we have
\begin{equation}\label{equ:billiardintersectionnumber}
    \iota(t_\gamma, t_\gamma) \leq \iota(\gamma, \gamma),
\end{equation}
where $t_\gamma$ is an unfolding of $\gamma$ on $X_P$.
\par
Indeed, let $(v_1, v_2)$ be a transverse pair of tangent vectors at a self-intersection point of $t_\gamma$ on $X_P$. If the projection $\pi : X_P \to P$ maps the vectors $v_1$ and $v_2$ to parallel directions in $P$, then there exist two subpaths of $t_\gamma$, tangent to $v_1$ and $v_2$ respectively, that are mapped to the same geodesic segment in the billiard path $\gamma$. This implies that the total length of $t_\gamma$ is strictly greater than that of $\gamma$.
However, according to Remark~\ref{rmk:relation}, this situation can occur only if $\gamma$ has odd period, which contradicts our assumption. Therefore, every transverse pair of $t_\gamma$ projects to a transverse pair of $\gamma$. 
\par
Furthermore, if two distinct transverse pairs in $t_\gamma$ were to project to the same transverse pair in $\gamma$, then by the same reason, $t_\gamma$ would have to be longer than $\gamma$, which is impossible under our assumptions. Hence, the map $\pi$ induces an injective map from the transverse pairs of $t_\gamma$ into those of $\gamma$, which proves the claim.
\par
For generalized diagonals in $P$, by~\eqref{equ:billiardintersectionnumber} and Theorem~\ref{thm:MAIN}, we have that
$N^{diag}(P,k)\le N^{sc}(X_P,k)\leq (3n-6)2^{s}$, where $s=\frac{20n(n-1)\sqrt{k} + 20n}{\delta}$.
\par
Moreover, if $P$ is of area $1$, $X_P$ is of area $2$. By Theorem~\ref{thm:MAIN2}, the metric length of $t_{\gamma}$ divided by the square root of the area of $X_P$ is bounded by 
$
\frac{40n(n-1)\sqrt{k}+40n}{\delta\sqrt{\pi}}
+
\frac{20n(n-1)\sqrt{k}+20n}{\delta^{3/2}\sqrt{2\pi}}
$. Hence, the metric length of $t_{\gamma}$ is bounded by 
$
\frac{40n(n-1)\sqrt{2k}+40\sqrt{2}n}{\delta\sqrt{\pi}}
+
\frac{20n(n-1)\sqrt{k}+20n}{\delta^{3/2}\sqrt{\pi}}
$.

Next, we consider periodic billiard paths. Combining Corollary~\ref{cor:countclosedgeodesics} with the inequality~\eqref{equ:billiardintersectionnumber}, we obtain $N^{per}(P,k) \leq N^{cg}(X,k) \leq (3n - 6) \cdot 2^s$, where $s = \frac{20n(n - 1)\sqrt{k} + 20n}{\delta}$.
\end{proof}


\begin{thebibliography}{100}

\bibitem{ACM}
J. Athreya, Y. Cheung, H. Masur.
\newblock {\em Siegel-Veech transforms are in $L^{2}$}.
\newblock {Journal of Modern Dynamics}, Volume 14, 1-19, 2019.

\bibitem{AthreyaMasur2024}
J.~S. Athreya and H.~Masur.
\newblock {\em Translation surfaces}.
\newblock Graduate Studies in Mathematics, Volume 242. American Mathematical Society, 2024.

\bibitem{BCGGM}
M. Bainbridge, D. Chen, Q. Gendron, S. Grushevsky, M. M\"{o}ller.
\newblock {\em Strata of $k$-differentials}.
\newblock {Algebraic Geometry}, Volume 6, Issue 2, 196–233, 2019.

\bibitem{Ba}
A. Basmajian.
\newblock {\em Universal length bounds for non-simple closed geodesics on hyperbolic surfaces}.
\newblock {Journal of Topology}, Volume 6, Issue 2, 513–524, 2013.

\bibitem{BaPV}
A. Basmajian, H. Parlier, H. Vo.
\newblock {\em The shortest non-simple closed geodesics on hyperbolic surfaces}.
\newblock {Mathematische Zeitschrift}, Volume 306, Issue 1, 8, 2024.

\bibitem{Chew}
P. Chew.
\newblock {\em There is a planar graph almost as good as the complete graph}.
\newblock {Proceedings of 2nd Symposium on Computational Geometry}, Yorktown Heights, NY, 1986.

\bibitem{Filip2024}
S.~Filip.
\newblock {\em Translation surfaces: dynamics and Hodge theory}.
\newblock {EMS Surveys in Mathematical Sciences}, Volume 11, Issue 1, 63-151. 2024.

\bibitem{FoxKer}
R. H. Fox, R. B. Kershner.
\newblock {\em Concerning the transitive properties of geodesics on a rational polyhedron}.
\newblock {Duke Math. J.}, Volume 2, 147–150, 1936.

\bibitem{Fu2024}
K. Fu.
\newblock{\em Uniform length estimates for trajectories on flat cone surfaces}.
\newblock To appear in {\em Annales de l'Institut Fourier}.
\newblock{ArXiv:2409.14188}, 2024.

\bibitem{Fu2025}
K. Fu.
\newblock {\em Siegel-Veech measures of convex flat cone spheres}.
\newblock To appear in {\em Compositio Mathematica}.
\newblock{ArXiv:2504.14731}, 2025.

\bibitem{Galperin1995}
G.~Galperin, T.~Krüger, and S.~Troubetzkoy.
\newblock \emph{Local instability of orbits in polygonal and polyhedral billiards}.
\newblock {Communications in Mathematical Physics}, 169:463–473, 1995.

\bibitem{HooperLowerBounds}
W.~P. Hooper.
\newblock {\em Lower bounds on growth rates of periodic billiard trajectories in some irrational polygons}.
\newblock Journal of Modern Dynamics, \textbf{1}(4):649--660, 2007.

\bibitem{Kat}
A. Katok.
\newblock {\em The growth rate for the number of singular and periodic orbits of a polygonal billiard}.
\newblock {Communications in Mathematical Physics},  Volume 111, 151-160, 1987.

\bibitem{kne}
H. Kneser.
\newblock {\em Geschlossene Fl\"{a}chen in dreidimensionalen Mannigfaltigkeiten}.
\newblock {Jahresbericht der Deutschen Mathematikver-Vereinigung},  Volume 38, 248-260, 1930.

\bibitem{Ma1}
H. Masur.
\newblock {\em Lower bounds for the number of saddle connections and closed trajectories of a quadratic differential}.
\newblock {Holomorphic Functions and Moduli I}, Mathematical Sciences Research Institute Publications, Springer, Volume 10, 225-228, 1988.

\bibitem{Ma2}
H. Masur.
\newblock {\em The growth rate of trajectories of a quadratic differential}.
\newblock {Ergodic Theory and Dynamical Systems}, Volume 10, 151-176, 1990.

\bibitem{MS91}
H. Masur and J. Smillie.
\newblock {\em Hausdorff dimension of sets of nonergodic measured foliations}.
\newblock {Annals of Mathematics}, Volume 134, Issue 3, 455-543,
1991.

\bibitem{MasurTabachnikov2002}
H.~Masur and S.~Tabachnikov.
\newblock {\em Rational billiards and flat structures}.
\newblock {Handbook of Dynamical Systems}, Volume 1A, pages 1015--1089. North-Holland, 2002.

\bibitem{McMullen2023}
C.~T. McMullen.
\newblock {\em Billiards and Teichmüller curves}.
\newblock Bulletin of the American Mathematical Society (New Series), \textbf{60}(2):195--250, 2023.

\bibitem{RSch06}
R.~E. Schwartz.
\newblock {\em Obtuse triangular billiards I: Near the $(2,3,6)$ triangle}.
\newblock {Experimental Mathematics}, \textbf{15}(2):161--182, 2006.

\bibitem{RSch08}
R.~E. Schwartz.
\newblock {\em Obtuse triangular billiards II: One hundred degrees worth of periodic trajectories}.
\newblock {Experimental Mathematics}, \textbf{18}(2):137--171, 2009.

\bibitem{Sch}
D.~Scheglov.
\newblock {\em Complexity growth of a typical triangular billiard is weakly exponential}.
\newblock {Journal d'Analyse Mathématique}, \textbf{142}:105--124, 2020.

\bibitem{Tabachnikov1995}
S.~Tabachnikov.
\newblock {\em Billiards}.
\newblock Panoramas et Synthèses, Volume 1. Société Mathématique de France, Paris, 1995.

\bibitem{Ta}
G. Tahar.
\newblock {\em Counting saddle connections in flat surfaces with poles of higher order}.
\newblock {Geometriae Dedicata}, Volume 196, Issue 1, 145-186, October 2018.

\bibitem{Ta1}
G. Tahar.
\newblock {\em Geometric Triangulations and Flips}.
\newblock {C. R. Math. Acad. Sci. Paris}, Volume 357, Issue 7, 620-623, 2019. 

\bibitem{Thu}
W. Thurston.
\newblock {\em Shapes of polyhedra and triangulations of the sphere}.
\newblock {Geometry and Topology Monographs}, Volume 1, 511-549, 1998.

\bibitem{tokarsky2018point}
G.~Tokarsky, J.~Garber, B.~Marinov, and K.~Moore.
\newblock {\em One hundred and twelve point three degree theorem}.
\newblock{ArXiv:1808.06667}, 2018.

\bibitem{Tr}
M. Troyanov.
\newblock {\em Les surfaces euclidiennes à singularités coniques}.
\newblock {Enseign. Math.}, Volume 32, 79-94, 1986.

\bibitem{Wright2015}
A.~Wright.
\newblock {\em Translation surfaces and their orbit closures: an introduction for a broad audience}.
\newblock EMS Surveys in Mathematical Sciences, \textbf{2}(1):63--108, 2015.

\bibitem{KZ}
A.~N. Zemljakov and A.~B. Katok.
\newblock {\em Topological transitivity of billiards in polygons}.
\newblock Matematicheskie Zametki, \textbf{18}(2):291--300, 1975.

\bibitem{Zor}
A. Zorich.
\newblock {\em Flat Surfaces}.
\newblock {Frontiers in Number Theory, Physics, and Geometry}, Volume 1, 439-585, 2006.
\end{thebibliography}
\end{document}